\documentclass[12pt,reqno]{amsart}
\usepackage{amssymb,amsmath,amsthm,amscd,enumerate,bbm}

\DeclareMathOperator{\sign}{sign}
\DeclareMathOperator{\Ran}{Ran}

\DeclareMathOperator{\spec}{spec}
\DeclareMathOperator*{\slim}{s-lim}

\renewcommand\Im{\text{\rm Im}\,}
\renewcommand\Re{\text{\rm Re}\,}
\newcommand\even{\text{\rm even}\,}
\newcommand\odd{\text{\rm odd}\,}
\newcommand{\abs}[1]{\lvert#1\rvert}

\newcommand{\diag}{\operatorname{diag}}

\newcommand{\q}{\quad}

\newcommand{\ti}{\tilde}
\newcommand{\wh}{\widehat}
\newcommand{\ov}{\overline}
\newcommand{\wt}{\widetilde}

\newcommand{\R}{{\mathbb R}}
\newcommand{\C}{{\mathbb C}}
\newcommand{\bbT}{{\mathbb T}}
\newcommand{\bbZ}{{\mathbb Z}}

\newcommand{\calH}{{\mathcal H}}

\newcommand{\vk}{\varkappa}

\numberwithin{equation}{section}

\theoremstyle{plain}
\newtheorem{theorem}{\bf Theorem}[section]
\newtheorem{lemma}[theorem]{\bf Lemma}

\newtheorem{assumption}[theorem]{\bf Assumption}
\newtheorem{corollary}[theorem]{\bf Corollary}

\theoremstyle{definition}
\newtheorem{definition}[theorem]{\bf Definition}

\theoremstyle{remark}
\newtheorem*{remark*}{\bf Remark}
\newtheorem{remark}[theorem]{\bf Remark}

\let\cal\mathcal

\let\Bbb\mathbb

\begin{document}

\title[Self-adjoint Hankel operators]{Spectral and scattering theory of self-adjoint Hankel operators
with piecewise continuous symbols}

\author{Alexander Pushnitski}
\address{Department of Mathematics,
King's College London, 
Strand, London, WC2R~2LS, U.K.}
\email{alexander.pushnitski@kcl.ac.uk}

\author{Dmitri Yafaev}
\address{Department of Mathematics, University of Rennes-1, 
Campus Beaulieu, 35042, Rennes, France}
\email{yafaev@univ-rennes1.fr}

\thanks{  Our collaboration has become possible through the hospitality and financial support 
of the Departments of Mathematics of the University of Rennes 1 and of KingÕs College London.
The London Mathematical Society Scheme 4 grant is also gratefully acknowledged. }

\begin{abstract}
We develop the spectral and scattering theory for  self-adjoint Hankel operators $H$ with piecewise continuous symbols. In this case every jump   of the symbol gives rise to a band of the absolutely continuous spectrum of $H$. 
We construct wave operators relating simple  ``model" (that is, explicitly diagonalizable) Hankel  operators for each jump and the given Hankel operator $H$. We show that the set of all these wave operators is asymptotically complete. This determines the absolutely continuous part of $H$. We also prove that    the singular continuous 
spectrum of $H$ is empty and that its eigenvalues may accumulate only to  ``thresholds"  in the absolutely continuous 
spectrum. All these results are reformulated in terms of Hankel operators realized as matrix or integral operators.
\end{abstract}



% \date{\today}
 
\maketitle

 %%%%%%%%%%%%%%%%%%%%%%%%%%%%%%%%%%%%%%%%%%%%%%% 
\section{Introduction} \label{sec.a}
%%%%%%%%%%%%%%%%%%%%%%%%%%%%%%%%%%%%%%%%%%%%%%%

%\centerline{\large Introduction} \bigskip
 
{\bf 1.1}.
Hankel  operators (see the books \cite{NK,Pe, Po}) admit various unitary equivalent descriptions. One of the common ones is the definition of Hankel operators $H$ in the Hardy space ${\Bbb H}_+^{2}({\Bbb T})\subset L^2 ({\Bbb T})$ of functions analytic inside the unit circle ${\Bbb T}$. Let $\omega \in L^\infty ({\Bbb T})$. Then for $f\in{\Bbb H}_+^{2}({\Bbb T})$, the function $(H f)(\mu)$,  $\mu\in {\Bbb T}$, is defined as the orthogonal projection in $ L^2 ({\Bbb T})$ of the function $\omega(\mu) f(\bar{\mu})$ onto the subspace ${\Bbb H}_+^{2}({\Bbb T})$. Of course, Hankel operators $H=H(\omega)$ with symbols $\omega \in L^\infty ({\Bbb T})$ are bounded.

 It is easy to see that $H$ is compact if  $\omega \in C({\Bbb T})$. On the contrary,   the jumps
 \begin{equation}
\varkappa (a)=\lim_{\varepsilon\to +0} \omega (ae^{i\varepsilon})-\lim_{\varepsilon\to +0} \omega (ae^{-i\varepsilon}),
\q a\in{\Bbb T},
\label{eq:XX1}\end{equation}
(one supposes here that the limits exist but are not equal) of the symbol yield  bands of the essential spectrum $\spec_{\rm ess} (H)$. To be more precise, it was shown by S.~R.~Power
in \cite{Po1, Po} that
 \begin{multline*}
\spec_{\rm ess} (H)=[0, (2i)^{-1}\varkappa (1)] \cup [0, (2i)^{-1}\varkappa (-1)] \cup
\\
  \bigcup_{\Im a_{j} > 0} [-(2i)^{-1} (\varkappa (a_{j}) \varkappa (\bar{a}_{j}))^{1/2}, (2i)^{-1} (\varkappa (a_{j}) \varkappa (\bar{a}_{j}))^{1/2} ]
 \end{multline*}
where $[\alpha,\beta]$ is the interval between the points $\alpha,\beta\in {\Bbb C}$ (we do not distinguish $[\alpha,\beta]$ and $[ \beta, \alpha]$). Note that the contribution of a jump at a complex point $a$ is nontrivial only if the symbol $\omega(\mu)$ has jumps at both points $a$ and $\bar{a}$.

Much more complete information (see subs.~3.4, for more detailes) can be obtained about the modulus $|H|=\sqrt{H^* H}$ of $H$. It is shown in \cite{PY1}  that the absolutely continuous (a.c.) 
spectrum\footnote{In the right-hand side of \eqref{eq:XX1x} as well as in all relations of this type, 
we use the convention that every term denotes a spectral band of multiplicity one 
in the a.c. spectrum.} of $|H|$
 \begin{equation}
\spec_{\rm ac} (|H|)=   \bigcup_{ a_{j}\in {\Bbb T}}\,  [0, 2^{-1} |\varkappa (a_{j}) | ].
\label{eq:XX1x}\end{equation}
It is assumed in \cite{PY1} that $\omega(\mu)$ has  finitely many jumps $a_{j}\in {\Bbb T}$ so that the union in \eqref{eq:XX1x} also has finite number of terms. Furthermore, the singular continuous spectrum of $|H|$ is empty and its eigenvalues different from $0$ and the points $ 2^{-1} |\varkappa (a_{j})|$ have finite multiplicities and may accumulate only to these points.

\medskip

{\bf 1.2.}
Our goal here is to develop the spectral and scattering theory for  {\it self-adjoint } Hankel operators with piecewise continuous symbols.  
In the self-adjoint case we have $\omega(\bar{\mu} )= \ov{\omega (\mu)}$. Therefore if $\omega(\mu)$ has a jump $\varkappa (a)$ at some point $a\in{\Bbb T}$, then it also has the jump $\vk(\bar{a})=-\ov{\varkappa(a)}$ at the point $\bar{a}$; in particular, $\vk(\pm 1)$ are necessarily imaginary numbers. We suppose that $\omega(\mu)$ has a finite number of jumps and that $\omega(\mu)$ is a continuous function away from its jumps. 
We assume that at every jump $a$, the one-sided limits $\omega(a_{\pm})$ of $\omega(\mu)$ as $\mu\to a_{\pm}= a e^{\pm i0}$ exist and satisfy the logarithmic H\"older continuity condition
\begin{equation}
 \omega(\mu)-\omega(a_{\pm}) = O (\big| \ln |\mu-a| \big| ^{-\beta_{0}}), \q \beta_{0}>0.
\label{eq:XX3}\end{equation}
This condition is   slightly stronger than the simple one-sided continuity of $\omega(\mu)$ but is of course weaker than the H\"older continuity. One of our main results (see Theorem~\ref{thm.g1a}) can be formulated as follows.

%\begin{theorem}\label{XX}
%Let $\omega(\mu)$ satisfy the assumptions above.

$1^0$ {\it  If    $\beta_{0}>1$, then 
\begin{equation}
\spec_{\rm ac} (H)= [ 0, (2i)^{-1} \varkappa (1)  ] \cup  [ 0, (2i)^{-1} \varkappa (-1)  ] \cup \bigcup_{ \Im a_{j} > 0} [-2^{-1} |\varkappa (a_{j}) |, 2^{-1} |\varkappa (a_{j}) | ] .
\label{eq:XX4}\end{equation}

  $2^0$  If   $\beta_{0}> 2$, then  the singular continuous spectrum of $H$ is empty and its eigenvalues distinct from $0$ and from the points $(2i)^{-1} \varkappa (1)  $,  $(2i)^{-1} \varkappa (-1)  $ and $ \pm 2^{-1} |\varkappa (a_{j})|$ where  $\Im a_{j} >  0$ have finite multiplicities and may accumulate only to these points. }
  
  \smallskip
  
  Relation \eqref{eq:XX4} shows that every real  jump or  a pair of complex conjugate jumps of the symbol $\omega(\mu)$ yields its own band of the multiplicity one  a.c. spectrum. Of course  if $\varkappa (1) =0$ or $\varkappa (-1) =0$, then the corresponding term  in \eqref{eq:XX4}  disappears.

   Under different assumptions on $ \omega(\mu)$,  relation \eqref{eq:XX4}  was obtained earlier by J.~S.~Howland  in \cite{Howl0}
 who used the trace class method in scattering theory. We rely on a multichannel scheme in the so-called smooth approach exposed in our preceding publication   \cite{PY}.  The results of \cite{PY} can be considered as a simplified version
of the famous Faddeev's solution  \cite{Fadd} of the three particle quantum problem.  In our case different channels of scattering are described in terms
  of    ``model''    operators for all points of discontinuity of $\omega$. Each model operator  $H_{+}$, $H_{-}$ and $H_{j} $  corresponds to jumps of $\omega$ at the points $+1$, $-1$ and $(a_{j}, \bar{a}_{j})$,  $\Im a_{j } >0$, and each of them yields one of the bands of the a.c. spectrum in the right-hand side of \eqref{eq:XX4}.

  Our conditions on  the symbol $\omega(\mu)$
  are much weaker than those of J.~S.~Howland. This is important for reformulations of our results 
 for Hankel operators realized as matrix and integral operators. 
   
\medskip

{\bf 1.3.}
Our approach shows  that the resolvent  of the operator $H$ sandwiched between appropriate weight functions has boundary values on the real axis except for the thresholds $0$,  $(2i)^{-1} \varkappa (1)  $,  $(2i)^{-1} \varkappa (-1)  $ and $ \pm 2^{-1} |\varkappa (a_{j})|$ for all complex points of jumps $a_j$. Results of this type are known as the limiting absorption principle. 

Another important ingredient of our approach is the construction of the   wave operators for model operators 
$H_{+}$, $H_{-}$, $H_{j} $ and the operator $H$. They play the same role as the wave operators for different channels of scattering in the three particle  problem. Thus   in the case of several jumps of the symbol, the scattering problem for Hankel operators becomes multichannel. Similarly to the three particle  problem, the ranges of  different  wave operators   are orthogonal to each other, and their orthogonal sum coincides with the a.c.  subspace  of the operator $H$. The last result is known as the asymptotic completeness of wave operators. It directly implies relation \eqref{eq:XX4} but also contains   information about the asymptotics of $\exp(-iHt)f$ as $t\to\pm \infty$.

The basis for the construction of model operators  is an explicit diagonalization of some simple Hankel operator.  We recall that Hankel operators can be realized as integral operators in the space $L^2 ({\Bbb R}_{+})$ with kernels ${\bf h}(t+s)$ which depend only on the sum of variables $t,s\in {\Bbb R}_{+}$. We rely on the Hankel operator  $\cal M$ with    ${\bf h}(t)=\pi^{-1}(t+2)^{-1}$ considered  by F. G. Mehler in \cite{Me}. Due to a slow decay of this function as $t\to \infty$ the operator $\cal M$ is not compact. Actually, it has the simple purely a.c. spectrum coinciding with the interval $[0,1]$. In terms of the symbol $\psi (\nu)$ (the Fourier transform of $(2\pi)^{1/2}{\bf h} (t)$), this fact is a consequence (cf. relation \eqref{eq:XX4}) of the jump $2i$ of $\psi (\nu)$ at the point $\nu=0$.

 Model operators $H_{+}$ and $H_{-}$ are directly constructed in terms of the   operator $\cal M$. The construction of 
 model operators $H_{j}$ for pairs of complex jumps is much  more complicated. The operators $H_{j}$ can also be realized as Hankel operators but their symbols are $2\times 2$ matrix valued functions so that they act in the direct sum
   ${\Bbb H}_+^{2}({\Bbb T}) \oplus {\Bbb H}_+^{2}({\Bbb T})$. 
   
    The construction of scattering theory for the pairs $(H_{+} ,H)$, $(H_{-} ,H)$ and $(H_{j} ,H)$ relies on the following arguments. We show that with an appropriate choice of model operators the symbol of the operator
   \[
\wt{H}= H-H_{+} -H_{-} -\sum_{\Im a_{j}>0} H_{j}
\]
%\label{eq:XX5}\end{equation}
has no jumps and so it is, in some sense, negligible. To be more precise, we establish a factorization $\wt{H}= Q^* \wt{K} Q$ where $\wt{K} $ is compact and $Q$ is smooth with respect to all operators $H_{+}$, $H_{-}$ and $H_{j}$. Roughly speaking, this means that $Q \vartheta\in L^2 ({\Bbb T})$ for all eigenfunctions $\vartheta$ (of the continuous spectra) of   $H_{+}$, $H_{-}$ and $H_{j}$. Similar factorizations are true for all   products $H_{+} H_{-}$, $H_{\pm} H_{j}$ and $H_{j} H_{k}$ where $j\neq k$. Here  we use the fact that the singularities of the symbols of the operators $H_{+}$, $H_{-}$ and $H_{j}$ are disjoint. These analytic results allow us to verify the assumptions of  \cite{PY}  which, in particular, yields  spectral results on the operator $H$ stated above.

Note that a similar scheme has been used by S.~R.~Power in his study of the essential spectrum of $H$ in \cite{Po1, Po} where it was however sufficient to verify the compactness of the operators $\wt{H}$, $H_{+} H_{-}$, $H_{\pm} H_{j}$ and $H_{j} H_{k}$ where $j\neq k$.  For the  study of the a.c. spectrum of $H$,  J.~S.~Howland required  in \cite{Howl0}  that these operators belong to the trace class.

An important issue in our approach is the choice of the class of smooth operators $Q$. As $Q$, we choose the operator of multiplication by a function $q(\mu)$ vanishing in singular points of the symbol $\omega(\mu)$. This choice is well adapted to the separation of   singularities of $\omega(\mu)$. Another possibility (see \cite{Y2}) is to choose for $Q$ the operator of multiplication by a function of $t$ tending to zero %logarithmically 
as $t\to 0$ and $t\to \infty$ in the realization of Hankel operators as integral operators in $L^2 ({\Bbb R}_{+})$.

\medskip

{\bf 1.4.}
Let us  now state our results for Hankel operators $\wh{H}$  in the space  ${\ell}_{+}^2= {\ell}^2({\Bbb Z}_{+})$ where $\wh{H}$ is defined by the formula
\begin{equation}
(\wh{H} u)_{n}= \sum_{m=0}^\infty h_{n+m} u_{m}, \q u= (u_{0}, u_{1},\ldots)\in {\ell}^2_{+}, \q h_n = \bar{h}_n, \q {\Bbb Z}_{+}.
\label{eq:HF}\end{equation}
Recall that $\wh{H}$ is compact if $h_{n} = o(n^{-1})$  as $n\to\infty$. On the other hand, if 
$h_{n}= \pi^{-1}(n+1)^{-1}$, then  $\wh{H}$ (the Hilbert matrix) has the simple a.c. spectrum coinciding with $[0,1]$.
We assume that
\begin{equation}
h_{n}= (\pi n)^{-1} \big(\kappa_{+} + (-1)^n \kappa_{-}  + 2 \sum_{j=1}^{N_{0}} \kappa_{j}\sin (n \theta_{j}-\varphi_{j})\big) + O \big(n^{-1}(\ln n)^{-\alpha_{0}}\big) 
\label{eq:F}\end{equation}
as $n\to\infty$. Here   $\theta_{j} $ are distinct numbers in $ (0,\pi)$; the phases $\varphi_{j}\in [0,\pi)$  and the amplitudes $\kappa_+, \kappa_-, \kappa_{j}\in{\Bbb R}$ are arbitrary. The main result in this case is (see Theorem~\ref{L}):

 \smallskip

$1^0$  {\it  If    $\alpha_{0}>2$, then 
\begin{equation}
\spec_{\rm ac} (\wh{H})= [ 0, \kappa_{+}  ] \cup  [ 0, \kappa_-  ] \cup  \bigcup_{ j=1}^{ N_{0}} [-\kappa_{j} , \kappa_{j}  ] .
\label{eq:F1}\end{equation}

  $2^0$  If   $\alpha_{0}> 3$, then  the singular continuous spectrum of $\wh{H}$ is empty and its eigenvalues different from $0$ and the points $\kappa_{+} $,  $\kappa_{-}  $ and $ \pm \kappa_{j}$  have finite multiplicities and may accumulate only to these points.}
  
   \smallskip
 
Observe that the right-hand side of \eqref{eq:F} contains oscillations with different frequencies.
Formula \eqref{eq:F1} shows that every term in asymptotics \eqref{eq:F}  yields its own channel of scattering and that 
 there is no ``interference" between different terms. 
 
Note that the condition $  O \big(n^{-1}(\ln n)^{-\alpha_{0}}\big)$ on matrix elements  does not guarantee (for any $\alpha_{0}$)  that the corresponding Hankel operator   belongs to the trace class. So  J.~S.~Howland's results \cite{Howl0} do not cover the case of matrix elements  with asymptotics \eqref{eq:F}.

\medskip

{\bf 1.5.}
The paper is organized as follows. In Section~2 we recall  the results of our preceding publication \cite{PY} on multichannel scheme in scattering theory. Auxiliary  information about  Hankel operators is collected in Section~3. This information is used to state our results in different  representations of Hankel operators.

Main results are stated and proven in Section~6. The proofs consist of 2 ingredients. The first is the construction of    ``model''    operators for all points of discontinuity of $\omega$. This is carried out  in     Section~4 for real points and in     Section~5 for pairs of complex conjugate points. Here
  the  smoothness of some class of operators with respect to model operators is also 
   verified.       The second  ingredient consists of compactness results   about  Hankel operators sandwiched by singular weights.  We borrow these results from our previous publication \cite{PY1}.
    All our results can be extended to Hankel operators acting in the Hardy space ${\Bbb H}^2_{+}({\Bbb R})$ and to operators
        with operator-valued symbols.

   Finally, in Section~7 we state  our results for Hankel operators realized    
 as infinite matrices in the space ${\ell}_{+}^2 $ and as integral operators in the space $L^2({\Bbb R}_{+})$.

%%%%%%%%%%%%%%%%%%%%%%%%%%%%%%%%%%%%%%%%%%%%%%%
%%%%%%%%%%%%%%%%%%%%%%%%%%%%%%%%%%%%%%%%%%%%%%%
\section{Multichannel scheme} \label{sec.b}
%%%%%%%%%%%%%%%%%%%%%%%%%%%%%%%%%%%%%%%%%%%%%%%
%%%%%%%%%%%%%%%%%%%%%%%%%%%%%%%%%%%%%%%%%%%%%%%

%Here we formulate the results of \cite{PY} which will then be applied to Hankel operators. \medskip

In the first two subsections, we collect  some background facts from   scattering theory; see, e.g., the book \cite{Yafaev},  for a detailed presentation. In the last subsection we recall the results of our paper \cite{PY} which will be used here.

\medskip

{\bf 2.1.}
Let $H $ be a self-adjoint operator in a Hilbert space $\calH$,
and let $E(\cdot )=E(\cdot; H)$ be  its spectral family. We denote by $  {\cal H}^{({\rm p})} (H)$ the subspace of $\cal H$ spanned by all eigenvectors of the operator $H$ and by $  {\cal H}^{({\rm ac})} (H)$ its a.c. subspace; $ P^{({\rm ac})} (H)$ is the orthogonal projector onto $  {\cal H}^{({\rm ac})} (H)$; $ H^{({\rm ac})} $ is the restriction of $H$ onto $  {\cal H}^{({\rm ac})} (H)$.

   Suppose that     the spectrum of the operator $H$ is  a.c.   and  has a constant (possibly infinite) multiplicity $n$ on a bounded open interval $\Delta\subset{\Bbb R}$. We consider   a unitary mapping  
\begin{equation}
 {\sf F } :E (\Delta){\cal H}\rightarrow L^2(\Delta;{\cal N})
 =L^2(\Delta)\otimes {\cal N} , \q \dim {\cal N}=n,
\label{eq:DIntF}
\end{equation}
of the subspace $E (\Delta){\cal H}$ onto the space of vector-valued functions of  $\lambda\in \Delta$  
with values in ${\cal N}$.  Assume that   this mapping transforms
$H$   into the  operator $\mathsf{A}_{\Delta}$ of multiplication  by $\lambda$ in the space $L^2(\Delta; {\cal N})$, that is,
\[
( {\sf F } H f)(\lambda)=  \lambda ({\sf F } f)(\lambda), \q
f\in E (\Delta){\cal H}, \q \lambda\in \Delta .
 \]
 Along with
$L^2(\Delta; {\cal N})$, we consider the space 
$C^\gamma(\Delta; {\cal N})$, $\gamma\in (0,1]$, of H\"older
continuous vector-valued functions. We set ${\sf F } f=0$ for $f\in E ({\Bbb R}\setminus\Delta){\cal H}$.

% We refer to the book \cite{Yafaev}  for a detailed discussion of the following definition.
  
\begin{definition}\label{strsm}
 Let $Q$ be a bounded   operator in the space $  {\cal H}  $. The operator $Q$ is called strongly $H$-smooth on an interval $\Delta$   with an exponent
$\gamma \in (0,1] $ if,  for some diagonalization $ {\sf  F }$ of the operator $E (\Delta) H$, the condition
\begin{equation}
 {\pmb|} ( {\sf  F}  Q^* f)(\lambda){\pmb|} \leq C\|f\| ,\quad
 {\pmb|} ( {\sf  F} Q^*  f)(\lambda')- ({\sf  F}  Q^*  f)(\lambda){\pmb|} \leq C |\lambda'-\lambda|^\gamma \|f\|
\label{eq:ZG}\end{equation}
is satisfied for all $f\in {\cal H}$.
Here the constant $C$ does not depend on $\lambda$ and $\lambda'$ in compact subintervals of $\Delta$.
\end{definition}

   Definition~\ref{strsm} depends   on the choice of mapping \eqref{eq:DIntF}, but in applications the operator ${\sf  F} $ emerges   naturally.
 For a strongly $H$-smooth operator $Q$, the operator $Z (\lambda;Q):{\cal H}\rightarrow {\cal N}$,
defined by the relation
\[
Z (\lambda; Q) f=({\sf  F}  Q^*  f)(\lambda), 
\]
  is bounded and depends H\"older continuously on $\lambda\in \Delta$.

 Assume that an operator $Q$ is strongly $H$-smooth. If an operator $B$ is bounded, then the product  $B Q$ is also strongly $H$-smooth. Let $U$ be a unitary operator in $\cal H$ and $\widetilde{H}= U^* H U$. Then the operator $\widetilde{Q}= Q U $  is  strongly $\widetilde{H}$-smooth for the diagonalization $\widetilde {\sf  F }=  {\sf F } U$ of $\widetilde{H}$.

 It is convenient to give also a ``global'' definition of  $H$-smoothness adapted to our purposes. 
 
 % without referring to a particular interval $\Delta$.

\begin{definition}\label{strsmg}
Suppose that, apart from the point spectrum $\spec_{p} (H)$, the spectrum of a self-adjoint operator $H$ is a.c., has a constant multiplicity and coincides with    the closure of a finite  union $\Delta$ of disjoint  open intervals $\Delta^{(l)}=(\alpha_{l},\beta_{l})$, $l=1,\ldots, L$. Assume also  that $\spec_{p} (H) \cap\Delta=\varnothing$. A bounded operator $Q$ is called strongly  $H$-smooth if it is  strongly  $H$-smooth on all intervals $\Delta^{(l)}$.
\end{definition}

Under the hypothesis of the above definition, the set $\cal T$  of thresholds  of the operator $H$ is defined as the collection of all end points $\alpha_{1}, \beta_{1},\ldots, \alpha_{L}, \beta_L$.

 \medskip

{\bf 2.2.}
Suppose that for self-adjoint operators $H_{0}$ and $H$,
the strong limits
        \[
\slim_{t\to\pm\infty} e^{i Ht} e^{- i H_0 t} P^{\rm(ac)} (H_{0}) =: W_{\pm} (H,H_0)    
\]
%\label{eq:W1}\end{equation}
  exist.  The operators $W_{\pm} (H,H_0)$ are known as wave operators. They are automatically isometric on the subspace ${\cal H}^{\rm(ac)} (H_{0})$, 
  enjoy the intertwining property  
 \[
 H W_{\pm} (H,H_0)=W_{\pm} (H,H_0)H_0 ,
 \]
 and their ranges $R (W_{\pm} (H,H_0) )\subset {\cal H}^{\rm(ac)} (H )$. 
 
 The wave operator $W_{\pm} (H,H_0)$ is called complete if   $R (W_{\pm} (H,H_0) )={\cal H}^{\rm(ac)} (H )$. The completeness of $W_{\pm} (H,H_0)$ is equivalent to the existence of $W_{\pm} (H_0, H)$; in this case  
  \[
 W_{\pm} (H_0,H) = W_{\pm}^* (H,H_0).
 \]
 
 Note also the multiplication theorem: if the wave operators $W_{\pm} (H,H_1)$ and $W_{\pm} (H_{1},H_0)$ exist, then the wave operator $W_{\pm} (H,H_0)$ also exists and
 \begin{equation}
 W_{\pm} (H,H_0)=W_{\pm} (H,H_1)W_{\pm} (H_{1},H_0).
\label{eq:mt}\end{equation}

\medskip

{\bf 2.3.}
Let $H_{1}$, \ldots, $H_{N}$ and $\wt H$ be bounded self-adjoint operators 
on a Hilbert space $\calH$. Our goal is to study the spectral properties of the operator
\begin{equation}
H=    H_{1} +\cdots + H_{N} + \wt{H} 
\label{eq:A}\end{equation} 
under certain smoothness assumptions on all  products $H_n H_m$, $n\neq m$, and on the operator $\wt{H} $. We suppose that all operators $H_{n}$, $n=1 , \ldots ,  {N}$, satisfy the conditions of Definition~\ref{strsmg} on a set $\Delta_{n}$ and that $0\not\in\Delta_{n}$. Note that in interesting cases $0 $ belongs to the closure of $\Delta_{n}$ at least for one $n$.
Let   $Q$ be a bounded operator on $\calH$ such that its kernel is trivial and its range $R(Q)$ is dense in  $\calH$.   We need the following

\begin{assumption}\label{assu} 
a. For all $n=1,\ldots, N$,  the operator $Q$ is  strongly $H_n$-smooth $($see Definitions~$\ref{strsm}$ and $\ref{strsmg})$  with an exponent $\gamma\in (0,1]$.

b. The operator $\wt{H}$ can be represented as
$\wt{H}=Q^* \wt{K} Q $ with a compact operator $\wt{K}$. 

c. For all $n,m\geq1$, $n\neq m$, 
the operators $H_n H_m$ can be represented as
\[
H_n H_m  =Q^*  K_{n,m} Q 
\]
%\label{eq:m}\end{equation}
where the operators $K_{n,m}$ are compact. 

d.  For all $n=1,\ldots,N$, the operators $Q  H_n Q^{-1}$ defined on the   set $R(Q)$ extend to  bounded operators. 
\end{assumption}

The spectral structure of the operator \eqref{eq:A}
 is described in the following assertion.   We denote by $ {\cal T}_n $ be the set of the  
 thresholds of the operator $H_n$ and put  ${\cal T}={\cal T}_{1}\cup\cdots \cup {\cal T}_{N}$. 
 Recall that by definition, ${\sf A}_{\Delta}$ is the operator of multiplication by the 
 independent variable in $L^2(\Delta)$. 

\begin{theorem}\label{main}\cite{PY}
Under Assumption~$\ref{assu}$ we have:

$1^0$    The  
operator $H^{\rm(ac)}$ is unitarily equivalent to the direct sum 
  \[
  {\sf A}_{\Delta_{1}}  \oplus  \cdots   \oplus    {\sf A}_{\Delta_N}  . 
  \]

  $2^0$ Suppose additionally that $\gamma>1/2$. Then the  singular continuous spectrum of $H$ is empty and the eigenvalues of   $H$   in  the set ${\Bbb R}\setminus {\cal T}$ have
  finite multiplicities and can   accumulate only  to the  set ${\cal T}$. 
  \end{theorem}
  
  The following assertion is known as
 the limiting absorption principle. 

\begin{theorem}\label{LAP}  \cite{PY}
Under Assumption~$\ref{assu}$ with $\gamma>1/2$, the operator-valued function $  Q (H-z)^{-1} Q^* $ is H\"older continuous with any exponent $\gamma'<\gamma$  in $z$ if $\pm\Im z\geq 0$, $\Re z \in {\Bbb R}\setminus {\cal T}$ away from  eigenvalues of $H$. 
\end{theorem}

The following assertion summarizes the
scattering theory for the set of  the operators $ H_{1}, \ldots,  H_{N} $
 and the operator $H$. 
 
%The following result was derived in \cite{PY} from Theorem~\ref{main}.

    \begin{theorem}\label{ScTh} \cite{PY}
  Under Assumption~$\ref{assu}$ we have:
    
    $1^0$ For all $ n=1,\ldots, N$, the wave operators $ W_{\pm} (H,H_n)$ exist.

 $2^0$ These operators enjoy the intertwining property  
 \[
 H W_{\pm} (H,H_n)=W_{\pm} (H,H_n) H_{n}, \q n= 1,\ldots, N. 
 \]
 The wave operators are isometric and their ranges are orthogonal to each other, that is,
 \[
\Ran W_{\pm} (H,H_n)\bot \Ran W_{\pm} (H,H_m),\q n\neq m.
\]

 $3^0$ The  asymptotic completeness holds:
  \[
\Ran  W_{\pm}  (H,H_1 ) \oplus  \cdots \oplus\Ran  W_{\pm}  (H,H_N )=  {\cal H}^{\rm(ac)} (H).
\]
%\label{eq:W6}\end{equation}
%where ${\cal H}^{(p)}$ is the subspace spanned by all eigenvectors of the operator $A$.
     \end{theorem}
     
%   Theorems~\ref{main}, \ref{LAP} and \ref{ScTh} are proven in \cite{PY}.  
   
   %  We mention that all the results concerning the absolutely continuous spectrum and the wave operators remain true if Assumption~\ref{assu}.b holds for some $\gamma>0$. The condition $\gamma>1/2$ is used only for the study of the singular component of the spectrum.
  
%%%%%%%%%%%%%%%%%
%\begin{corollary}\label{scth}
%%%%%%%%%%%%%%%%%
% The restriction of the  operator $A  $ on the subspace ${\cal H}\ominus  {\cal H}^{(p)}$ is unitarily equivalent to the orthogonal sum $ A_1 \oplus \cdots \oplus A_N$.
%\label{eq:orth}\end{equation}\end{corollary}

 %%%%%%%%%%%%%%%%%%%%%%%%%%%%%%%%%%%%%%%%%%%%%%%
%%%%%%%%%%%%%%%%%%%%%%%%%%%%%%%%%%%%%%%%%%%%%%%              
\section{Hankel   operators}\label{sec.e}
%%%%%%%%%%%%%%%%%%%%%%%%%%%%%%%%%%%%%%%%%%%%%%%
%%%%%%%%%%%%%%%%%%%%%%%%%%%%%%%%%%%%%%%%%%%%%%%

%\subsection{Hankel operators}

Here we collect standard information on various representations (see  the diagrams below)  of
 Hankel operators $H$. Observe that  Hankel operators are always defined     by the same formula
\begin{equation}
H =P_{+}\Omega J P_{+}^*,
\label{eq:HA}
\end{equation}
but the definitions of the operators $P_{+}$, $\Omega$ and $ J$ depend on the representation. We will consider four representations and state our results in all of them. It is convenient to  keep in mind all representations because some results obvious in one of them are difficult to see in others.

%All of them are important for us because a unitary transformation which is quite explicit in one representation does not necessarily admit an efficient description in other representations.

\medskip

{\bf 3.1.}
Let us begin with the representation of Hankel operators in the  Hardy space 
${\Bbb H}^2_{+} ({\Bbb T})\subset L^2 ({\Bbb T})$ 
of functions analytic in the unit disc $\mathbb D$ 
(for the precise definitions of Hardy classes, see, e.g., the book \cite{Hof}). 
The norm in the space $L^2 ({\Bbb T})$ is defined in a standard way by
\[
\| f\|_{L^2 ({\Bbb T})}^2 =  \int_{\Bbb T} |f(\mu)|^2    dm(\mu),\q dm(\mu)=(2\pi i \mu)^{-1}d\mu.
\]
Note that $dm(\mu)$ is 
the Lebesgue measure on $\bbT$ normalized so that $m(\bbT)=1$.
In formula \eqref{eq:HA},
 $P_{+}: L^2 ({\Bbb T}) \to {\Bbb H}^2_{+} ({\Bbb T})$ is the orthogonal projection, 
$J=J^*: L^2 ({\Bbb T}) \to L^2 ({\Bbb T})$ is the involution,
\[
(J f)(\mu)=\bar{\mu} f(\bar{\mu}),
\quad \mu\in\mathbb T.
\]
Obviously, $J$ maps ${\Bbb H}^2_{+} ({\Bbb T})$  onto ${\Bbb H}^2_{-} ({\Bbb T})$  
where ${\Bbb H}^2_{-} ({\Bbb T})={\Bbb H}^2_{+} ({\Bbb T})^\bot$ is the space of functions 
analytic in $ \C\setminus \overline{\mathbb D}$  and decaying at infinity. The operator $P_{-}=J P_{+} J$ is the orthogonal projection of $L^2 ({\Bbb T})$ onto $ {\Bbb H}^2_- ({\Bbb T})$.

The operator of multiplication 
$\Omega : L^2 ({\Bbb T}) \to L^2 ({\Bbb T})$ 
is defined by the formula
\begin{equation}
(\Omega f)(\mu)=\mu \omega(\mu) f( \mu ),
\quad \mu\in\mathbb T.
\label{eq:HA1}\end{equation}
The function $\omega(\mu)$ is always assumed to be bounded. 
Thus operator \eqref{eq:HA} is determined by the  function $\omega(\mu)$, 
that is, $H=H (\omega)$. 
The function $\omega(\mu)$ is known as the symbol of the Hankel operator
 $ H (\omega)$. 
 Of course the symbol is not unique because $ H (\omega_{1})
 = H (\omega_2)$ if (and only if) $\omega_{1}-\omega_{2}\in {\Bbb H}^\infty_{-} ({\Bbb T})$.

% We recall that the projections $P_{+}$ and $P_{-}=I-P_{+}$ are
 % defined by the formula
%\begin{equation}
%(P_{\pm} f)(\mu)= \mp  \int_{\Bbb T} f(\mu') (\mu\bar{\mu}'-   1\mp 0)^{-1}    d m(\mu').
%\label{eq:P+}\end{equation}

We set ${\ell}^2={\ell}^2(\bbZ)$ and ${\ell}^2_+={\ell}^2(\bbZ_{+})$ where $\bbZ_{+}= \{0,1,2,\dots\}$. 
The unitary mapping ${\cal F}: L^2 ({\Bbb T})\to {\ell}^2$ corresponds 
to expanding  a function in the Fourier series:
\begin{equation}
\hat{f}_{n}=({\cal F}f)_{n} =  \int_{\Bbb T} f(\mu)    \mu^{-n } d m(\mu)
\label{eq:CF}\end{equation}
so that for a sequence $\hat{f}=\{\hat{f}_{n}\}$, $n\in \bbZ$, 
\begin{equation}
f(\mu) = ({\cal F}^* \hat{f}) (\mu) = \sum_{n=-\infty}^\infty \hat{f}_{n}\mu^n.
\label{eq:CF1}\end{equation}
Then $\widehat{P}_{+}={\cal F} P _{+}{\cal F}^*: {\ell}^2\to {\ell}^2_{+}$ is the orthogonal projection onto the subspace ${\ell}^2_{+}$.
% that is,   \[ (\widehat{P}_{+}\hat{f})_n = \hat{f}_n \q\mathrm{for}\q n\geq 0 \q\mathrm{and}\q 
%(\widehat{P}_{+}\hat{f})_n = 0 \q\mathrm{for}\q n< 0 .\]
The operators $\widehat{J} ={\cal F}J{\cal F}^*: {\ell}^2\to {\ell}^2 $ and
$\widehat{\Omega} ={\cal F} \Omega  {\cal F}^*: {\ell}^2\to {\ell}^2 $ act by  the formulas
\[
(\widehat{J} \hat{f})_{n}= \hat{f}_{-n-1}
\]
and
\[
(\widehat{\Omega} \hat{f})_{n}= \sum_{m=-\infty}^\infty \hat{\omega}_{n-m-1}\hat{f}_{m}
\]
%\label{eq:CF2}\end{equation}
where $\hat{\omega}_{n } $ are the Fourier coefficients of the function $\omega (\mu)$.
According to \eqref{eq:HA},  
this leads to the standard definition of the Hankel operator 
\[
\widehat{H} ={\cal F} H {\cal F}^* = \widehat{P}_{+}\widehat{\Omega}\widehat{J}\widehat{P}_{+} ^*: {\ell}^2_{+}\to {\ell}^2_{+} 
\]
%\label{eq:CF2a}\end{equation}
 by the formula
\[
(\widehat{H} \hat{f})_{n}= \sum_{m=0}^\infty \hat{\omega}_{n+m}\hat{f}_{m}.
\]
%\label{eq:CF2A}\end{equation}

\medskip

{\bf 3.2.}
Recall that the mapping
\[
\mu=\frac{\nu -i/2}{\nu+i/2}
\]
%\label{eq:frlin}\end{equation}
of ${\Bbb R}$ onto ${\Bbb T}$ can be extended to the conformal mapping from the upper half-plane onto the unit disc. The unitary operator  ${\cal U}: L^2 ({\Bbb T})\to  L^2 ({\Bbb R} )$ corresponding to this mapping is defined by the equality
\begin{equation}
({\cal U} f)(\nu)= (2 \pi)^{-1/2} (\nu+i/2)^{-1} f \bigl(\tfrac{\nu-i/2}{\nu+i/2}\bigr).
\label{eq:un}\end{equation}
Since
\[
\nu=\frac{i }{2}\frac{1+\mu}{1-\mu},
\]
%\label{eq:zy}\end{equation}
we have
\begin{equation}
({\cal U}^* {\bf f})(\mu)=  i (2 \pi)^{1/2}
 (1-\mu)^{-1} {\bf f} \bigl(\tfrac{i}{2} \tfrac{1+\mu}{1-\mu}\bigr).
\label{eq:un1}\end{equation}
Observe that  ${\cal U}: {\Bbb H}_{\pm}^2 ({\Bbb T})\to {\Bbb H}_{\pm}^2 ({\Bbb R} )$   and that ${\bf P}_{\pm} = {\cal U}P_{\pm} {\cal U}^*$ are the orthogonal projections onto the Hardy classes ${\Bbb H}_{\pm}^2 ({\Bbb R}  )$.  Set ${\bf J}= - {\cal U} J {\cal U}^*$, ${\pmb \Omega}= - {\cal U} \Omega {\cal U}^*$. Then
$({\bf J}{\bf f}) (\nu)=  {\bf f} (-\nu)$
and
\[
({\pmb \Omega}{\bf f}) (\nu)=\psi(\nu) {\bf f}  (\nu) 
\]
%\label{eq:pw}
where
\begin{equation}
\psi(\nu)= - \tfrac{\nu-i/2}{\nu+i/2}\: \omega\big(\tfrac{\nu-i/2}{\nu+i/2}\big).
\label{eq:pwu}
\end{equation}
As  always,
\begin{equation}
{\bf H}= {\bf H}(\psi)= {\cal U} H (\omega) {\cal U}^* = {\bf P}_+ {\pmb \Omega}{\bf J}{\bf P}_{+}^*.
\label{eq:pww}
\end{equation}

The last, fourth, representation is obtained by applying the Fourier transform $\Phi$: 
%L^2 ({\Bbb R}; d\nu) \to L^2 ({\Bbb R}; d t)$,
\[
\hat{\bf f} (t)=(\Phi {\bf f})(t)= (2\pi)^{-1/2} \int_{-\infty}^\infty {\bf f}(\nu) e^{-i \nu t} d\nu.
\]
Then $\widehat{\bf P}_{\pm}= \Phi {\bf P}_{\pm}\Phi^*$ acts as the multiplication 
by the characteristic function $\mathbbm{1}_{\pm}  $ of the half-axis ${\Bbb R}_{\pm}$, that is,
\[
(\widehat{\bf P}_{\pm}\hat{\bf f} )(t)= \mathbbm{1}_{\pm}(t) \hat{\bf f} (t).
\]
%\label{e1}
In this representation,
\[
(\widehat{\bf J} \hat {\bf f}) (t)=  (\Phi {\bf J} \Phi^*\hat {\bf f}) (t)= \hat {\bf f} (-t)
\]
%\label{e2}
%\end{equation}
and $\widehat{\pmb \Omega}=\Phi^* {\pmb \Omega} \Phi$ is the convolution:
\begin{equation}
(\widehat{\pmb \Omega} \hat {\bf f} )(t)= (2\pi)^{-1/2}\int_{-\infty}^\infty  \hat { \psi} (t -s )\hat {\bf f}  (s )ds.
\label{e3}
\end{equation}
Then the Hankel operator
\begin{equation}
\widehat{\bf H} =\Phi{\bf H} \Phi^*= \widehat{\bf P}_{+}\widehat{\bf \Omega} \widehat{\bf J}\widehat{\bf P}_{+}^*
\end{equation}
 acts in the space $L^2 ({\Bbb R}_{+})$ by the standard formula
\begin{equation}
(\widehat{\bf H} \hat{\bf f})(t)= (2\pi)^{-1/2} \int_{0}^\infty \hat{\psi}(t+ s)\hat{\bf f}(s)ds.
\label{eq:K}
\end{equation}
In general, for $\psi \in L^\infty ({\Bbb R})$, formulas \eqref{e3} and \eqref{eq:K} should of course be understood in the sense of distributions. The function $\psi$ is known as the symbol of the Hankel operator $\widehat{\bf H} $.

\medskip

{\bf 3.3.}
Finally, we note that the representations in the spaces ${\ell}^2_{+}$ and $L^2({\Bbb R}_{+})$ are   connected by the operator ${\cal L}=\Phi  {\cal U} {\cal F}^* $.  It can  be  directly expressed 
in terms of the Laguerre functions, but we do not need the corresponding formulas in this paper.     

  The relations between different representations can be summarized by the following diagrams:  
 \[
   \begin{CD}
   L^2 ({\Bbb T})    @> {\cal U}>>  L^2 ({\Bbb R}; d\nu)
   \\
    @VV {\cal F}  V     @   VV \Phi V     
   \\
    {\ell}^2   @> {\cal L}>>  L^2 ({\Bbb R}; dt)
        \end{CD}
        \quad\q \q
        \begin{CD}
   {\Bbb H}^2_{+} ({\Bbb T})    @> {\cal U}>>   {\Bbb H}^2_{+} ({\Bbb R})
   \\
    @VV {\cal F}  V     @   VV \Phi V     
   \\
     {\ell}^2_{+}   @> {\cal L}>>  L^2 ({\Bbb R_{+}})
    \end{CD} 
    \] 
   %      \label{eq:dia}\end{equation}
     $$ 
   \begin{CD}
   f(\mu)   @>>> {\bf f}(\nu)= ({\cal U}f)(\nu)
   \\
    @VV    V     @   VV   V     
   \\
    \hat{f}_{n} = ({\cal F} f)_{n}  @>>>   \hat{\bf f} (t)=(\Phi {\bf f}) (t) 
    \end{CD}
   $$
   and
   $$
    \begin{CD}
 H   @>>> {\bf  H } = {\cal U}  H {\cal U}^*
   \\
    @VV    V     @   VV   V     
   \\
    \widehat{H}  ={\cal F} H {\cal F}^* @>>>   \widehat{\bf H}  =\Phi  {\bf  H }\Phi^*
    \end{CD}
   \q\q \q
   \begin{CD}
   \omega (\mu)   @>>> \psi(\nu) 
   \\
    @VV    V     @   VV   V     
   \\
    \hat{\omega}_{n}    @>>>   \hat{\psi} (t)      
    \end{CD}
    $$

    A Hankel operator $H$ is self-adjoint in $ {\Bbb H}^2_{+} ({\Bbb T})  $ if $ J \Omega^*=   \Omega J$, i.e., 
      \begin{equation}
    \omega(\bar{\mu} )=\overline{\omega(\mu)}.
    \label{eq:HA2}\end{equation}
    This equality transforms into relations $\overline{\hat{\omega}_{n}}=\hat{\omega}_{n}$, $ \psi(-\nu)=\overline{\psi(\nu)}$ and 
    $\overline{\hat{ \psi} (t )}=\hat{ \psi}(t )$ in the spaces ${\ell}^2_{+}$,   ${\Bbb H}_{+}^2 ({\Bbb R})$ and $L^2 ({\Bbb R}_{+})$, respectively.

All these definitions   can be naturally extended to operators \eqref{eq:HA}
acting on functions taking values in an auxiliary Hilbert space $\cal N$. For example, formula \eqref{eq:HA} remains meaningful for an operator $H :   {\Bbb H}^2_{+}({\Bbb T})\otimes \cal N\to   {\Bbb H}^2_{+}({\Bbb T})\otimes \cal N$ if the operator $\Omega$ is defined by equality \eqref{eq:HA1} where $\omega(\mu): \cal N\to \cal N$  is an operator-valued function.

\medskip

{\bf 3.4.}
We systematically use the following elementary trick. Instead of the operator \eqref{eq:HA} in the Hardy space ${\Bbb H}_{+}^2 ({\Bbb T})$ we consider the operator $P_{+} \Omega J P_{+}$ acting in the  space $L^2 ({\Bbb T})$. This operator has the same spectrum as the operator \eqref{eq:HA} except for the additional zero eigenvalue of infinite multiplicity. We usually use the same notation $H$ for both of these operators. In particular, this trick allows us to use freely the results of \cite{PY1}. 
All  results are stated for Hankel operators in ${\Bbb H}_{+}^2 ({\Bbb T})$ while all proofs are given in the  space $L^2 ({\Bbb T})$.

Let us now come back to the results about the modulus of $|H|$ of Hankel operators $H= P_{+} \Omega J P_{+}$ stated in subs.~1.1.
We proceed from the following result of \cite{PY1}.

 \begin{theorem}\label{sym}
 Suppose that a function $\omega:\bbT\to\C$   is  continuous 
apart from some jump discontinuities at finitely many points    $a_j \in {\Bbb T} $ 
with jumps \eqref{eq:XX1}. At every point of discontinuity $a_{j} \in \bbT$,   assume condition \eqref{eq:XX3}. Then 
the  a.c. spectrum  of the operator 
 \begin{equation}
H_{\rm sym} =P_{+}\Omega P_{-} + P_{-} {\Omega}^*  P_{+} 
\label{eq:XX1xsym}\end{equation}
acting in the  space $L^2 ({\Bbb T})$ is given by the relation
 \[
\spec_{\rm ac} (H_{\rm sym})=   \bigcup_{ j}\,  [-2^{-1} |\varkappa (a_{j}) | , 2^{-1} |\varkappa (a_{j}) | ].
\]
Furthermore, the singular continuous spectrum of $H_{\rm sym}$ is empty and its eigenvalues different from $0$ and the points $ \pm 2^{-1} |\varkappa (a_{j})|$ have finite multiplicities and may accumulate only to these points. 
 \end{theorem}
 
It is easy to see that Theorem~\ref{sym} implies the results  about  the operator  $|H|$. Indeed, it follows from \eqref{eq:XX1xsym} that
\begin{multline}
H_{\rm sym}^2 =P_{+}\Omega P_{-} {\Omega}^* P_{+}  + P_{-} {\Omega}^*  P_{+} \Omega P_{-}
=P_{+}\Omega J P_{+} J {\Omega}^* P_{+}  + J P_{+} J  {\Omega}^*  P_{+} \Omega J P_{+} J 
\\  = H H^* + J H^* H J.
\label{eq:SYM}\end{multline}
According to Theorem~\ref{sym} the a.c. spectrum of the operator $H_{\rm sym}^2$ consists of the union of the intervals $[0 , 4^{-1} |\varkappa (a_{j}) |^2 ]$ (with every interval counted twice). The singular continuous spectrum of $H_{\rm sym}^2$ is empty and its eigenvalues different from $0$ and the points $ \pm 4^{-1} |\varkappa (a_{j})|^2$ have finite multiplicities and may accumulate only to these points. 
Let us further use that the non-zero parts of the operators $ H H^* $ and $ J H^* H J$ are unitarily equivalent and that they act in the orthogonal subspaces ${\Bbb H}_+^{2}({\Bbb T})$ and ${\Bbb H}_-^{2}({\Bbb T})$, respectively. Therefore  the results  about the operator  $|H|$  stated in subs.~1.1 follow from the identity \eqref{eq:SYM}.

     % 
%%%%%%%%%%%%%%%%%%%%%%%%%%%%%%%%%%%%%%%%%%%%% 
\section{Model operators for jumps at real points}\label{sec.fmo}
%%%%%%%%%%%%%%%%%%%%%%%%%%%%%%%%%%%%%%%%%%%%%   
%%%%%%%%%%%%%%%%%%%%%%%%%%%%%%%%%%%%%%%%%%%%%

Here we construct   ``model'' operators $H_{+}$, $H_{-}$ corresponding to jumps of the symbol at the points $1$, $-1$.  The operators  $H_{\pm}$ will  be directly diagonalized (see Theorem~\ref{smD}) with the help of the results on the Mehler operator  discussed in subs.~4.1. Then we find (see Theorem~\ref{sm}) a class of operators smooth with respect to $H_{+}$ and $H_{-}$.

%The construction of the   operators $K_a$ is more complicated. It is given in the representation $H_{+}^2({\Bbb T})\otimes{\Bbb C}^2$ described in subs.~2.2. 

\medskip

{\bf 4.1.}
Following \cite{PY1}
as an ``elementary model'' operator, we choose the Mehler operator defined in the space $L^2 ({\Bbb R}_{+})$ by the formula 
\begin{equation}
(  {\cal M} u) (t)=   \pi^{-1}\int_{0}^\infty (2+ t+s)^{-1} u(s)ds.
\label{eq:Meh}
\end{equation}
The spectral decomposition of $ {\cal M}$ is well known and
is based on Mehler's formula:
\begin{equation}
\int_0^\infty \frac{P_{-\frac12+i\tau}(1+s)}{2+t+s} d s
=
\frac\pi{\cosh(\pi \tau)}
P_{-\frac12+i\tau}(1+t),
\quad
t, \tau\in\R_+,
\label{f1}
\end{equation}
where $P_{-\frac12+i\tau}(x)$ is the Legendre function. It can be defined     (see formulas (2.10.2) and (2.10.5) in the book \cite{BE}) for all $x>1$  in terms of the hypergeometric function
$F(a,b,c;z) $ and the gamma-function
  $\Gamma (\cdot)$ as
 \[
P_{-\frac12+i\tau}(x)
=
\Re \Big(\frac{\Gamma( i\tau)}{\sqrt{\pi}\Gamma(\tfrac12+ i\tau)}
2^{\tfrac12+i\tau}
F(\tfrac14-i\tfrac{\tau}{2},\tfrac34-i\tfrac{\tau}{2};1-i\tau; x^{-2})
x^{-\frac12 + i\tau}\Big).
\]
%\label{f16}\end{equation}
It follows that $P_{-\frac12+i\tau}(x)$ is a smooth function of $x > 1$ and it has the asymptotics
  \begin{equation}
P_{-\frac12+i\tau}(x)
=
\Re \big( \frac{\Gamma( i\tau)}{\sqrt{\pi}\Gamma(\tfrac12+ i\tau)}
2^{\tfrac12+i\tau} x^{-\frac12+ i\tau}\big)
+ O(x^{-5/2}), \q x\to\infty ,
\label{eq:f12}
\end{equation}
which is differentiable in $x$ and $\tau$. Moreover, the functions
$P_{-\frac12+i\tau}(x)$ and $P_{-\frac12+i\tau}'(x)$ are bounded as $x\to 1$.

  The Mehler-Fock transform 
$\Psi$ 
is  defined   by  the  formula
\begin{equation}
(\Psi f) (\tau)= \sqrt{\tau\tanh(\pi \tau)}
 \int_0^\infty  P_{-\frac12+i\tau}(t+1)f(t)  dt   
\label{f2}
\end{equation}
where $f\in C_0^\infty(\R_+)$ (see, e.g., \S 3.14 of  \cite{BE}).
Then   formula \eqref{f1} can be written as 
\begin{equation}
(\Psi  {\cal M} f)(\tau)
=
\frac1{\cosh(\pi \tau)} (\Psi f)(\tau), 
\quad \tau>0.
\label{f3}
\end{equation}
A detailed proof of the following assertion can be found in \cite{Yaf1}.

%%%%%%%%%%%%%%%%%%
\begin{lemma}\label{lma.f4}
%%%%%%%%%%%%%%%%%%
Let  the    operator $\Psi$ be defined
by formula \eqref{f2}. Then       $\Psi  $ is a unitary operator in $L^2({\Bbb R}_{+})$  and formula \eqref{f3} holds. In particular, the Mehler operator $\cal  M$ has the purely a.c.
spectrum $[0,1]$ of multiplicity one. 
\end{lemma}

 \begin{remark}\label{rma.f4}
Instead of the Mehler operator, for similar purposes,
  J.~S.~Howland  used    the Hankel operator 
 with   kernel ${\bf h}(t)=\pi^{-1}e^{-t}t^{-1}$ in  \cite{Howl0}. This operator was diagonalized  by W.~Magnus and M.~Rosenblum in \cite{Ma, Ro}. It also has the simple purely
   a.c.  spectrum coinciding with the interval $[0,1]$, but now the kernel is singular at the point $t=0$,  and the corresponding symbol $\psi (\nu)$ has a jump at infinity. 
 \end{remark}

\medskip
 
{\bf 4.2.}
The Mehler operator is of course a Hankel operator and,
as is well known, its  symbol   can be chosen as a smooth  function with one jump discontinuity. In order to define this symbol, consider the function   
\begin{equation}
\zeta (\nu)
=
\frac{1}{\pi}
\int_0^\infty \frac{\sin(\nu t)}{2+t}d t ,\q \nu\in{\Bbb R}.
\label{H4}\end{equation}
Obviously, this function is real and odd.  Since
\[
 \zeta  (\nu)
=
\frac{1}{\pi} \Im \Big(
e^{-2i\nu } \int_{\nu}^\infty \frac{e^{2 i x}}{x}d x \Big), \q \nu>0,
\]
the function $\zeta  \in C^\infty (\R\setminus\{0\})$ and  $\zeta (\nu)$ admits an asymptotic expansion in powers $\nu^{-2k-1}$, $k=0,1,\ldots $ as $\nu\to\infty$. Moreover, the limits  $\zeta  (\pm 0)$  exist,  $\zeta ( \pm 0)=\pm 1/2$ and $\zeta ' (\nu)=O(|\ln|\nu||)$ as $\nu\to 0$. Calculating the Fourier transform of   function \eqref{H4}, we find that  
\[
\hat{\zeta }(t)=   \frac{-i} { \sqrt{2\pi}} \frac{\sign t}{2+| t| } , \q t\in{\Bbb R}.
\]
%\label{f0}\end{equation}
Thus the symbol of the operator $\Phi^* {\cal  M} \Phi$ equals $2 i \zeta(\nu)$ and hence
\begin{equation}
{\cal M} = \Phi     {\bf H}( 2i \zeta)   \Phi^* .
\label{eq:f33}
\end{equation}

These  results can of course be transplanted onto the unit circle. Let us set
\begin{equation}
v (\mu)= - 2  i\mu^{-1}  \zeta \bigl(\tfrac{i}{2}\tfrac{1+\mu}{1-\mu}\bigr),
\label{eq:om}
\end{equation}
and let $H(v)$ be the Hankel operator on ${\Bbb H}^2_{+}(\Bbb T)$ with this symbol.
Note that $v \in C^\infty (\Bbb T\setminus\{-1\})$ and the limits 
$v (-1\mp i0)= \pm i $ exist so that   the jump of $v$ at the point $-1$ equals
\[
 v (-1- i0)-v (-1 + i0) = 2  i.
 \]
Comparing formulas \eqref{eq:pwu} and  \eqref{eq:om}, we see that the symbol of the operator $ {\cal U}   H(v)  {\cal U} ^*$ also equals $ 2 i \zeta (\nu)$.
Hence according to \eqref{eq:f33} we have
\begin{equation}
{\cal M} = \Phi {\cal U}   H(v) {\cal U} ^* \Phi^* .
\label{eq:f3}
\end{equation}
Putting together equalities \eqref{f3} and \eqref{eq:f3}, we arrive at the following assertion.

%%%%%%%%%%%%%%%%%%
\begin{lemma}\label{lme}
%%%%%%%%%%%%%%%%%%
Let  the symbol $v (\mu)$ be defined
by formulas \eqref{H4} and \eqref{eq:om}. Then 
\[
(F H(v) f)(\tau)=   \frac1{  \cosh(\pi \tau)} (F   f)(\tau) , \q \tau>0,
\]
%\label{eq:Me1x}\end{equation}
where $f\in {\Bbb H}^2_{+}(\Bbb T)$ is arbitrary and 
\begin{equation}
F= \Psi \Phi {\cal U}: {\Bbb H}^2_{+}(\Bbb T)\to L^2 ({\Bbb R}_{+} )
\label{eq:Me1}\end{equation}
is the unitary operator.
\end{lemma}

Thus  the operator $ H (v)$ reduces to the operator of multiplication by the   function $  ( \cosh(\pi \tau))^{-1}$
in the space $L^2 ({\Bbb R}_{+})$. Making additionally the change of variables $\lambda=  (\cosh(\pi \tau))^{-1}$, we can further reduce the operator $ H(v)$ to the operator of multiplication by the independent variable $\lambda$ in the space $L^2 (0, 1)$. However, diagonalization \eqref{eq:Me1} is quite convenient for our purposes.

Using the operator $H(v)$ whose symbol \eqref{eq:om} is singular at the point $\mu =-1$, it is easy to construct a model operator for the singularity at the point $\mu =1$. Let us set
  \begin{equation}
  v_{+}(\mu)=   v  (-\mu), \q v_{-}(\mu)=   v  ( \mu);
  \label{eq:pm}\end{equation}
 the functions  $v_{\pm}(\mu)$ have the jump $ 2 i$ at the points $\mu=\pm 1$. Note that  $H (v_{+}) =    {\sf R}  H ( v ) {\sf R}^*$ where ${\sf R}$
 is the reflection operator in ${\Bbb H}_+ ^2 ({\Bbb T})$ defined by the formula 
   \begin{equation}
   ({\sf R} f)(\mu)=f(-\mu).
   \label{eq:Re}\end{equation}
Set  
  \begin{equation}
  F_{+}=F {\sf R} \q, F_{-}=F .
   \label{eq:ReF}\end{equation}
 It follows from Lemma~\ref{lme} that the operator $F_{+} H(v_+) F_{+}^*$ acts in the space $ L^2 ({\Bbb R}_{+} )$ as multiplication by the function $ \big( \cosh(\pi \tau) \big)^{-1}$. In particular, we  obtain the following assertion.
 
%  Thus the following assertion is a direct consequence of  Lemmas~\ref{lme} and \ref{ME2}. 
  
  %%%%%%%%%%%%%%%
\begin{theorem}\label{smD}
%%%%%%%%%%%%%%%
Let  the symbols $ v_{\pm}  $ be defined by equalities \eqref{H4},  \eqref{eq:om} and   \eqref{eq:pm}.    The operators $H(v_\pm)$ have the purely  a.c.    simple spectrum coinciding with the interval $[0,1]$. 
 \end{theorem}

\medskip

{\bf 4.3.}
Define the operators $F_{\pm}$
 by formulas \eqref{eq:Me1} and \eqref{eq:ReF}.
In this subsection we consider $H(v_\pm)$ and $F_\pm$ as the operators in $L^2 ({\Bbb T})$ extending them by zero onto ${\Bbb H}_+^2 ({\Bbb T})^{\bot}$. Let us  construct some operators that are smooth (see Definition~\ref{strsm} where now $ {\cal N}={\Bbb C}$) with respect to $H(v_\pm)$.

For $a\in {\Bbb T}$, we  introduce a function on ${\Bbb T}$ by the equations
\begin{equation}
 q_{a}  (\mu)= \big|\ln |\mu - a | \big|^{-1 }\q {\rm for} \q |\mu - a | \leq e^{-1}
\label{eq:Xs}
\end{equation}
and $ q_{a}  (\mu)=1$ for $ |\mu - a | \geq e^{-1}$. Note that  $q_{a} \in L^\infty ({\Bbb T})$ and $q_{a}  (\mu)$ vanishes (logarithmically) only at one point $a\in \Bbb T$. Let the operator $Q_{a} $  in $ L^2  (\Bbb T) $ be defined  by the formula
\begin{equation}
( Q_{a} f) (\mu)=q_{a} (\mu) f(\mu)  .
\label{eq:X}\end{equation}

Our goal now is   to check the following result.

% $H (v)$-smoothness   (see Definition~\ref{strsm}) of  the operator $Q_{-1}=Q_{-1}^{(\beta)}$
%on the interval $\Delta=(0,1)$ for $ {\cal N}={\Bbb C}$. As an operator $F$ diagonalizing $H (v)$, we choose operator    \eqref{eq:Me1} extended by zero onto ${\Bbb H}_+^2({\Bbb T})^{\bot}$. 

 %%%%%%%%%%%%%%%
\begin{theorem}\label{sm}
%%%%%%%%%%%%%%%
Let  the symbols $ v_{\pm}  $ be defined by equalities \eqref{H4},  \eqref{eq:om} and   \eqref{eq:pm}, and let the operators $Q_{\pm 1} $ be defined by formula \eqref{eq:X}.    Then   the operator  $Q_{\pm 1}^{\beta} $ for $\beta>1/2$  is strongly $H(v_\pm)$-smooth on the interval $(0,1)$ for the diagonalization $F_{\pm}$    with any exponent   $\gamma<\beta-1/2$.
 \end{theorem}

Let us start with an informal interpretation of the result of Theorem~\ref{sm}.  By formula \eqref{f1} up to a normalization, eigenfunctions (of the continuous spectrum) of the operator $\cal M$ equal $\vartheta_{\tau}(t)=P_{-\frac12+i\tau}(1+t)$. In view of \eqref{eq:f12},  they do not belong to $L^2$ at infinity.  This implies   that the eigenfunctions $(\Phi^* \vartheta_{\tau})(\nu)$ of the operator $\Phi^*{\cal M}Ê\Phi$  do not belong to $L^2$ in a neighbourhood of  the point $\nu= 0$ and hence eigenfunctions    $ ({\cal U}^* \Phi^* \vartheta_{\tau})(\mu)$ of the operator $H (v)$  do not belong to $L^2$ in a neighbourhood of  the point $\mu= -1$. However Lemma~\ref{lma.f5} below shows that
the singularities  of the functions $ ({\cal U}^* \Phi^* \vartheta_{\tau})(\mu)$ at the point $\mu= -1$ are  rather weak.  This is an indication that an operator of multiplication by a bounded function  is $H (v)$-smooth provided this function logarithmically vanishes at the point $\mu=- 1$. We note that the
symbol  $v(\mu)$ has a jump at the point $\mu=- 1$.

%Following the convention of subs.~3.4, below we consider the operator $H(v)$ in $L^2 ({\Bbb T})$ extending it by zero onto ${\Bbb H}_+^2({\Bbb T})^{\bot}$.
 
The  formal proof  requires some   elementary information on the Fourier transform of the Legendre functions.

%%%%%%%%%%%%%%%%%%
\begin{lemma}\label{lma.f5}
%%%%%%%%%%%%%%%%%%
The integral 
\begin{equation}
 w_{\tau}(\nu)=  (2\pi)^{-1/2}  \int_1^\infty  
  P_{-\frac12+i\tau}(x )e^{ - i\nu x} dx , \q \tau>0,  
\label{eq:Xx1}\end{equation}
converges for all $\nu>0$, and it  is differentiable in $\tau$.
If  $\Delta\subset {\Bbb R}_{+}$ is a compact interval and $\tau\in\Delta$, 
then there exists $C=C (\Delta)$ such that
\begin{equation}
\abs{ w_{\tau}(\nu)}
\leq 
C  \abs{\nu}^{-1},
\quad
| \partial w_{\tau} (\nu)/ \partial \tau |
\leq
C  \abs{\nu}^{-1},
\quad 
\abs{\nu}\geq 1/2,
\label{f4}\end{equation}
and 
\begin{equation}
\abs{ w_{\tau}(\nu)  }
\leq 
C  \abs{\nu}^{-1/2},
\quad 
| \partial w_{\tau} (\nu)/ \partial \tau |
\leq
C  \abs{\nu}^{-1/2}\big|\ln |\nu | \big|,
\q
\abs{\nu}\leq 1/2, \quad \nu\not=0.
\label{f5}
\end{equation}
Moreover, the integral
\begin{equation}
   \int_1^ R 
  P_{-\frac12+i\tau}(x )e^{ - i\nu x} dx  
\label{eq:Xx1f}\end{equation}
is bounded by $C  \abs{\nu}^{-1}$ for $\abs{\nu}\geq 1/2$ and by $C  \abs{\nu}^{-1/2}$ for $\abs{\nu}\leq 1/2$ with a constant $C$ that does not depend on $R\leq\infty$.
\end{lemma}

 Estimates \eqref{f4} and \eqref{f5} are proven in  \cite{PY1}; see Lemma~3.10.  The assertion about the integral \eqref{eq:Xx1f} can be obtained in exactly the same way.

 Next, we derive a convenient representation for the operator $\Psi\Phi$. 
 According to our convention of subs.~3.4 we put $\Psi \hat{\bf f} =0$ for  $\hat{\bf f} \in L^2 ({\Bbb R}_{-})$ and consider $\Psi\Phi$ as a mapping of $L^2 ({\Bbb R})$ onto $L^2 ({\Bbb R}_{+})$.  Denote by ${\cal S}={\cal S} ({\Bbb R})$ the Schwartz space.

%%%%%%%%%%%%%%%%%%
\begin{lemma}\label{XY}
%%%%%%%%%%%%%%%%%%
Let ${\bf g}\in  {\cal S}  $.  Then 
\begin{equation}
(\Psi\Phi {\bf g})(\tau)=  \sqrt{\tau\tanh(\pi \tau)} \psi(\tau)
\label{eq:Fh}\end{equation}
  where 
\begin{equation}
\psi (\tau)= 
  \int_{-\infty}^\infty   {\bf g}(\nu)
w_{\tau}(  \nu) e^{ i \nu}    d\nu 
 \label{eq:Xx}\end{equation}
 and $w_{\tau}(  \nu)$ is function \eqref{eq:Xx1}.
 \end{lemma}

\begin{proof}
 It follows from \eqref{f2}   that 
\[
(\Psi\Phi {\bf g}) (\tau)= (2\pi)^{-1/2} \sqrt{\tau\tanh(\pi \tau)} \lim_{R\to \infty}
\int_0^R  d t
  P_{-\frac12+i\tau}(t +1)  
  \big(\int_{-\infty}^\infty   {\bf g}(\nu)
  e^{- i \nu t}    d\nu \big).
\]
Changing the order of integrations by the Fubini theorem, we obtain representation \eqref{eq:Fh} with
\[
\psi (\tau)= (2\pi)^{-1/2} \lim_{R\to \infty}
  \int_{-\infty}^\infty   {\bf g}(\nu)
     \big( \int_0^{R} 
  P_{-\frac12+i\tau}(t+1 )e^{ - i\nu t} d t \big)    d\nu . 
\]
Using the assertion of Lemma~\ref{lma.f5}   about integral \eqref{eq:Xx1f},
we can pass here to the limit by the dominated convergence  theorem.
 \end{proof}

Now we are in a position to complete the proof of Theorem~\ref{sm}.

%This will require some estimates on the Fourier transform
%of  the functions $\vartheta_{\tau}(t)$.

% which are the Legendre functions.  Note that the function $q_{-1}  (\mu)$ equals zero at the point $\mu=-1$ where the symbol of the operator $K (v )$ is not continuous and, as we shall see below,  its eigenfunctions are singular. 

\begin{proof}
Consider, for example, the sign $``-"$.
 We have to check the estimates
\begin{equation}
|( {  F}  Q_{-1}^\beta  f) (\tau)| \leq C \| f \|_{L^2(\Bbb T)}
\label{eq:X1}
\end{equation}
and
\begin{equation}
| ({  F}  Q_{-1}^\beta   f) (\tau')-  ({  F}  Q_{-1}^\beta   f) (\tau ) | \leq C |\tau'- \tau |^\gamma \| f \|_{L^2(\Bbb T)},\q \gamma<\beta-1/2, 
\label{eq:X2}\end{equation}
for    $\tau$ and $ \tau'$ in compact subintervals of ${\Bbb R}_{+}$ and all $f\in L^2 (\Bbb T)$. We can of course assume that $f$ belongs to the set ${\cal U}^* {\cal S}$ dense in $L^2 (\Bbb T)$.

 Put  $g= Q_{-1}^\beta f$, ${\bf f}={\cal U} f$, ${\bf g}={\cal U}g$ and $ {\bf q}   (\nu)= \big|\ln |\nu  | \big|^{-1 }$ for $ |\nu   | \leq e^{-1} $,
 $  {\bf q}   (\nu) =1$ for $ |\nu | \geq e^{-1}$.  Then using   notation  \eqref{eq:Fh}, we can equivalently rewrite estimates \eqref{eq:X1} and \eqref{eq:X2}  as
\begin{equation}
| \psi(\tau) |Ê\leq C \|{\bf q}^{-\beta}{\bf g}\| 
 \label{eq:Xx2}\end{equation}
 and
 \begin{equation}
| \psi(\tau')- \psi (\tau)|Ê\leq C  |\tau'- \tau |^\gamma\|{\bf q}^{-\beta}{\bf g}\| ,\q \gamma<\beta-1/2.
 \label{eq:XY2}\end{equation}
 
  By the Schwarz inequality, it follows from representation \eqref{eq:Xx} that
  \[
| \psi(\tau) |Ê\leq \|{\bf q}^{\beta}w_{\tau}\|  \|{\bf q}^{-\beta}{\bf g}\| 
\]
where ${\bf q}^{\beta}w_{\tau}\in L^2 ({\Bbb R})$ if $\beta>1/2$ by virtue of Lemma~\ref{lma.f5}. This proves \eqref{eq:Xx2}. Estimate
\eqref{eq:XY2} can be obtained quite similarly if one takes into account that $\abs{ w_{\tau'}(\nu)  -w_{\tau}(\nu)}$ is bounded 
by $C |\tau'-\tau |    \abs{\nu}^{-1 }$ for $|\nu| \geq 1/2$ and by $C |\tau'-\tau |^ \gamma   \abs{\nu}^{-1/2}\big|\ln |\nu | \big|^ \gamma $ for $|\nu| \leq 1/2$ according to estimates \eqref{f4} and \eqref{f5}, respectively. 
\end{proof}

   % 
%%%%%%%%%%%%%%%%%%%%%%%%%%%%%%%%%%%%%%%%%%%%% 
\section{Model operators for jumps at complex points}\label{sec.fmoc}
%%%%%%%%%%%%%%%%%%%%%%%%%%%%%%%%%%%%%%%%%%%%%   
%%%%%%%%%%%%%%%%%%%%%%%%%%%%%%%%%%%%%%%%%%%%%

Here we construct    ``model'' operators $H_a$ for   pairs $a,\bar{a} $ of complex conjugate points.   Although the operators $H_a$  are not Hankel in the space ${\Bbb H}^2_{+}({\Bbb T})$, they can be realized as Hankel operators  in the space of two components vector valued functions.
 In the first subsection, we describe the necessary  representation of ${\Bbb H}^2_{+}({\Bbb T})$.
 The operators  $H_a$  are diagonalized in Theorem~\ref{sm1D} and the class of $H_a$-smooth operators is found in Theorem~\ref{sm1}. 

% $H_{+}^2({\Bbb T})\otimes{\Bbb C}^2$

\medskip
 
{\bf 5.1.}
Let us identify the spaces $L^2 ({\Bbb T})\otimes {\Bbb C}^2$ and 
 $L^2 ({\Bbb T})$ and, in particular, their subspaces ${\Bbb H}^2_{+}({\Bbb T})\otimes {\Bbb C}^2$ and 
 ${\Bbb H}^2_{+}({\Bbb T})$. Put 
$f_\even   =2^{-1} (I+ {\sf R} )f$,  $f_\odd   =2^{-1} (I- {\sf R} )f$ where ${\sf R}$ is operator   \eqref{eq:Re}. Evidently, $f_\even   $ and  $f_\odd    $
     are the even and odd parts of $f$. Using notation \eqref{eq:CF}, \eqref{eq:CF1}, we set
   \begin{equation}
    \begin{split}
f^{(+)} (\mu) : & = f_\even (\mu^{1/2}) = \sum_{n=-\infty}^\infty \hat{f}_{2n}\mu^n,
\\
f^{(-)} (\mu) : & = \mu^{-1/2} f_\odd (\mu^{1/2}) = \sum_{n=-\infty}^\infty \hat{f}_{2n+1}\mu^n.
  \end{split}
\label{eq:CF3y}\end{equation}
Then\footnote{The upper index ``$\top$" means that a vector is regarded as a column.} 
 \begin{equation}
{\sf U}: f(\mu)\mapsto (f^{(+)}(\mu), f^{(-)}(\mu))^\top =: \vec{f}(\mu)
\label{eq:CF3}\end{equation}
 is a unitary mapping of  $L^2 ({\Bbb T}) $ onto  $L^2 ({\Bbb T}) \otimes {\Bbb C}^2$ and of 
  ${\Bbb H}^2_{+} ({\Bbb T}) $ onto  ${\Bbb H}^2_{+} ({\Bbb T}) \otimes {\Bbb C}^2 $.  Obviously, the operator
  ${\sf U}^*: L^2 ({\Bbb T}) \otimes {\Bbb C}^2\to   L^2 ({\Bbb T}) $ acts by the formula
  \[
  ({\sf U}^*  \vec{f})(\mu)= f^{(+)}(\mu^2)+ \mu f^{(-)}(\mu^2)   .
  \]
  
  A Hankel operator $ {\sf H }={\sf H }({\Sigma})$ in the space ${\Bbb H}^2_{+}({\Bbb T})\otimes {\Bbb C}^2$ is defined by formula \eqref{eq:HA} where  $  (J \vec{f})(\mu)= \bar{\mu}    \vec{f}( \bar{\mu})$, $  ( \Omega  \vec{f})(\mu)= \mu \Sigma(\mu)  \vec{f}(\mu)$
  and the symbol
     \begin{equation}
 \Sigma(\mu) =   \begin{pmatrix}  \sigma_{1,1} (\mu ) & \sigma_{1,2} (\mu )
 \\ 
\sigma_{2,1} (\mu )
& \sigma_{2,2} (\mu )
 \end{pmatrix} .
\label{eq:CF4}\end{equation}
  is a $2\times 2$ matrix-valued function.

    An easy calculation shows that, for a Hankel operator $   H(\omega):  {\Bbb H}^2_{+} ({\Bbb T}) \to {\Bbb H}^2_{+} ({\Bbb T})$,  the operator
\[
  {\sf H } (\Sigma):= {\sf U}H (\omega){\sf U}^*: {\Bbb H}^2_{+} ({\Bbb T}) \otimes {\Bbb C}^2 \to {\Bbb H}^2_{+} ({\Bbb T}) \otimes {\Bbb C}^2 
\]
% \label{eq:CF4x}\end{equation}
   is also the   Hankel operator     with    symbol \eqref{eq:CF4} where
        \begin{equation}\begin{split}
\sigma_{1,1} (\mu ) = &\omega_\even (\mu^{1/2}),\q \sigma_{2,2} (\mu )   =\mu^{-1}
 \omega_\even (\mu^{1/2}), 
 \\
 & \sigma_{1,2} (\mu )= \sigma_{2,1} (\mu )=
  \mu^{-1/2}\omega_\odd (\mu^{1/2})  .
\end{split} \label{eq:CF4n}\end{equation}
On the other hand,   for a Hankel operator ${\sf H} (\Sigma) $ in the space ${\Bbb H}^2_{+} ({\Bbb T})\otimes {\Bbb C}^2$, the operator $ {\sf U}^*{\sf H }(\Sigma) {\sf U}=: H (\Sigma)$ is   Hankel in the space  ${\Bbb H}^2_{+} ({\Bbb T}) $ if and only if $ \sigma_{1,1}(\mu)= \mu \sigma_{2,2}(\mu)$ and $ \sigma_{1,2}(\mu)=  \sigma_{2,1}(\mu)$. In this case the symbol $\omega(\mu)$ of this operator can be constructed by formulas \eqref{eq:CF4n}.

\medskip

{\bf 5.2.}
 First, we construct a model operator   corresponding to the  pair $(i,-i )$.  Clearly, the symbol 
     \begin{equation}
\omega_\varphi(\mu)=  (\sin \varphi - \mu \cos\varphi) v (\mu^2), \q  \varphi\in [0,2\pi),
  \label{eq:si}\end{equation}
  is smooth everywhere except the points $\pm  i$  where it has the jumps  $\pm 2 e^{\pm i\varphi}$.
 % \begin{equation}
%\varkappa =   \kappa e^{i\varphi }=   r + is , \q \kappa > 0,\q   r  , s \in{\Bbb R},
%\label{eq:HO2q}\end{equation}
%  and $-\bar{\varkappa}$, respectively.
Although the function $ \omega_\varphi(\mu)$ looks simple, we do not know how to diagonalize the Hankel operator $H(\omega_\varphi)$ explicitly.    Therefore we distinguish its singular part which will be done in the representation ${\Bbb H}^2_{+} (\Bbb T)\otimes {\Bbb C}^2$.
   It follows from formulas   \eqref{eq:CF4} and \eqref{eq:CF4n} that the symbol of the Hankel operator  $ {\sf U}H (\omega_\varphi) {\sf U}^*= :{\sf H}  (\Sigma_\varphi)$ acting  in the space ${\Bbb H}^2_{+} (\Bbb T)\otimes {\Bbb C}^2$ is   the matrix-valued function
     \begin{equation}
\Sigma_\varphi(\mu) =\Sigma_\varphi^{(0)}(\mu)+ \widetilde{\Sigma}_\varphi(\mu)
\label{eq:CF5}\end{equation}
  where
      \begin{equation}
 \Sigma_\varphi^{(0)}(\mu) =  \begin{pmatrix} \sin \varphi  &- \cos\varphi
 \\ 
- \cos\varphi
&- \sin \varphi  
 \end{pmatrix} v (\mu)
\label{eq:CF6}\end{equation}
and
 \begin{equation}
 \widetilde{\Sigma}_\varphi (\mu) =    \begin{pmatrix}  0 &0
 \\ 
0
& \sin \varphi 
 \end{pmatrix} (1+\bar{\mu} ) v (\mu)  .
\label{eq:CF6b}\end{equation}
 Note that  the symbol $ \Sigma_\varphi^{(0)}(\mu)$ has a jump at the point $\mu=-1$ while the symbol $ \widetilde{\Sigma}_\varphi(\mu)$  is a Lipschitz continuous function. Equality \eqref{eq:CF5} implies that
   \begin{equation}
H (\omega_\varphi) = {\sf U}^*{\sf H}  (\Sigma_\varphi^{(0)}){\sf U} +
 {\sf U}^*{\sf H}  (\widetilde{\Sigma}_\varphi){\sf U}. 
\label{eq:LL}\end{equation}

Diagonalizing the $2\times 2$ matrix in the right-hand side of \eqref{eq:CF6}, we see that
   \[
 \Sigma_\varphi^{(0)}(\mu) = v (\mu)   Y_{\varphi}^*
    \begin{pmatrix} 1  & 0 
 \\ 
0
&-1
 \end{pmatrix} 
 Y_{\varphi}  
\]
% \label{eq:DD1}\end{equation}
where
 \begin{equation}
Y_{\varphi}=
\frac{1}{\sqrt{2}} 
\begin{pmatrix} \frac{\cos\varphi}{\sqrt{1-\sin\varphi}}& - \sqrt{1-\sin\varphi} 
 \\ 
\frac{\cos\varphi}{\sqrt{1+\sin\varphi}} 
&\sqrt{1+\sin\varphi} 
 \end{pmatrix}.
\label{eq:DD2}\end{equation}
 It follows that
 \[
{\sf H} ( \Sigma_\varphi^{(0)}) =   Y_{\varphi}^*   \begin{pmatrix} H(v)  & 0 
 \\ 
0
&- H(v)  
 \end{pmatrix} Y_{\varphi} .
 \]
% \label{eq:DD3}\end{equation}
Thus  Lemma~\ref{lme} yields  the  explicit diagonalization of the operator  ${\sf H} ( \Sigma_\varphi^{(0)}) $. In particular,   the following result is a direct consequence of  Theorem~\ref{smD}.

  %%%%%%%%%%%%%%%
\begin{lemma}\label{sm1xD}
%%%%%%%%%%%%%%%
The operators ${\sf H} ( \Sigma_\varphi^{(0)})$
  and hence $H (\Sigma_\varphi^{(0)}) : = {\sf U}^*{\sf H} (\Sigma_\varphi^{(0)}){\sf U}$ have the purely a.c. simple spectrum coinciding with the interval $[-  1,  1]$.
     \end{lemma}

 Note that in the   particular case  $\sin\varphi =0$, we have  $\Sigma_\varphi(\mu)=\Sigma_\varphi^{(0)}(\mu)$ so that
  the operator $H(\omega_\varphi )$ can be  explicitly diagonalized.
  
  % If $\sin\varphi =0$, then  $\Sigma_\varphi(\mu)=\Sigma_\varphi^{(0)}(\mu)$.   
% If $\cos\varphi =0$, then the $2\times 2$ matrix \eqref{eq:CF5} is diagonal.   Such operators can be treated similarly to   \cite{KM}.

\medskip
 
{\bf 5.3.}
The case of  jumps  at  arbitrary    pairs $(a, \bar{a})$, $\Im a>0$, 
of complex conjugate points of $\Bbb T$ can be reduced to the case $a=i$. 
To that end, we consider the transformation of symbols under special fractional linear   maps of the unit disc corresponding to dilations of the upper half-plane.

 %%%%%%%%%%%%%%%
\begin{lemma}\label{conf}
%%%%%%%%%%%%%%%
Put   
  \begin{equation}
 (T_{\alpha} f)(\mu)= \sqrt{1-\alpha^2} (1+\alpha\mu)^{-1}f\big(\frac{\mu+ \alpha}{1+\alpha \mu}\big),\q \alpha\in (-1,1).
\label{eq:conf}\end{equation}
 Then $T_{\alpha}$ is   the unitary operator in ${\Bbb H}^2_{+}({\Bbb T})$ and, for an arbitrary Hankel operator   $H(\omega)$, we have
  \begin{equation}
  T_{\alpha}^* H(\omega) T_{\alpha}= H(\omega^{(\alpha)})
  \label{eq:conf1}\end{equation}
  where 
  \begin{equation}
\omega^{(\alpha)}(\mu)= \mu^{-1} \frac{\mu- \alpha}{1-\alpha \mu}\omega\big(\frac{\mu- \alpha}{1-\alpha \mu}\big).
\label{eq:conf2}\end{equation}
   \end{lemma}
   
   \begin{proof}
   It is convenient to make calculations in the representation ${\Bbb H}^2_{+}({\Bbb R})$.  It follows from formulas \eqref{eq:un} and \eqref{eq:un1} that the operator ${\bf T}_{\alpha}={\cal U}T_{\alpha} {\cal U}^*$ is the dilation:
      \[
 ({\bf T}_{\alpha} {\bf f})(\nu)= \sqrt{\frac{1+\alpha}{1-\alpha}}\,  {\bf f}\big(\frac{1+\alpha}{1-\alpha}\nu\big) .
\]
Therefore for a Hankel operator $ {\bf H} (\psi)$ with symbol $\psi(\nu)$, we have
\begin{equation}
{\bf T}_{\alpha}^* {\bf H} (\psi){\bf T}_{\alpha}={\bf H} (\psi_{\alpha})
\label{eq:conf3}\end{equation}
where
  \begin{equation}
   \psi^{(\alpha)}(\nu)=    \psi \big(\frac{1-\alpha}{1+\alpha}\nu\big).
 \label{eq:conf3z}\end{equation}

Recall that the symbols of the operators ${\bf H} (\psi)={\cal U} H(\omega) {\cal U}^*$ and $H(\omega)$ are linked by formula \eqref{eq:pwu}. 
Hence using \eqref{eq:conf3}, \eqref{eq:conf3z} and
making pull back to  ${\Bbb H}^2_{+}({\Bbb T})$, we obtain formulas \eqref{eq:conf1},  \eqref{eq:conf2}.
     \end{proof}
     
     For a pair $(a, \bar{a})$ where $a=e^{i\theta}$, $\theta\in (0,\pi)$, we now put
      \begin{equation}
\omega_{\varphi, \theta}(\mu)= \mu^{-1} \frac{\mu- \alpha}{1-\alpha \mu}\omega_{\varphi}\big(\frac{\mu- \alpha}{1-\alpha \mu}\big)
\label{eq:conf4}\end{equation}
where $\omega_{\varphi}(\mu)$ is symbol \eqref{eq:si}  and
  \[
\alpha= \frac{a-i}{1-ia}= \tan(\pi/4 - \theta/2).
\]
%\label{eq:rot}\end{equation}
Observe that the transformation $\mu\mapsto  (\mu- \alpha) (1-\alpha \mu)^{-1}$ sends the pair $(a, \bar{a})$ into $(i,-i)$.
Therefore the symbol  $ \omega_{\varphi, \theta}(\mu)$ is smooth everywhere except the points $a$ and $\bar{a}$ where it has the jumps   $2 i  e^{i(\varphi-\theta)}$ and $2 i  e^{i (\theta-\varphi)}$, respectively. It follows from Lemma~\ref{conf} that
\[
H(\omega_{\varphi, \theta}) = T_{\alpha}^* H(\omega_{\varphi}) T_{\alpha} \q {\rm where} \q\alpha=   \tan(\pi/4 - \theta/2).
\]
%  \label{eq:conf5}\end{equation}
%  where $a$ and $\alpha$ are linked by formula \eqref{eq:rot}. 
  
  Putting together this equality with \eqref{eq:LL}, we find that
  \begin{equation}
H(\omega_{\varphi, \theta}) = H_{\varphi, \theta}^{(0)}+  \wt{H}_{\varphi, \theta} 
  \label{eq:conf6}\end{equation}
  where
    \begin{equation}
 H_{\varphi, \theta}^{(0)} = T_{\alpha}^* {\sf U}^* {\sf H} (\Sigma_{\varphi }^{(0)}){\sf U} T_{\alpha}
  \label{eq:conf7}\end{equation}
  and
  \begin{equation}
 \wt{H}_{\varphi, \theta} = T_{\alpha}^* {\sf U}^* {\sf H} (\wt{\Sigma}_{\varphi } ){\sf U} T_{\alpha}.
   \label{eq:conf7B}\end{equation}
  Of course Lemma~\ref{sm1xD} yields the explicit diagonalization of the operator $ H_{\varphi, \theta}^{(0)}$.
  
 Let us summarize the results obtained.
  
  %%%%%%%%%%%%%%%
\begin{theorem}\label{sm1D}
%%%%%%%%%%%%%%%
Let the symbol $\omega_{\varphi, \theta}$ be defined by formulas \eqref{eq:si} and \eqref{eq:conf4}, and let  the symbols $ \Sigma^{(0)}_{\varphi   }$, $ \wt{\Sigma}_{\varphi   }$ be defined by formulas \eqref{eq:CF6}, \eqref{eq:CF6b}. 
Then:

$1^0$ Equalities \eqref{eq:conf6}  -- \eqref{eq:conf7B}  hold.

 $2^0$ The operator $ H_{\varphi, \theta}^{(0)}$ has the  purely a.c.    simple spectrum coinciding with the interval $[- 1 , 1  ]$. 
 \end{theorem}
  
  \medskip
 
{\bf 5.4.}
Similarly to subs.~4.3,  we now  extend all operators from ${\Bbb H}_+^2 ({\Bbb T})$ onto  $L^2 ({\Bbb T})$ setting them to zero   on ${\Bbb H}_+^2 ({\Bbb T})^{\bot}$. 
 We again use Definition~\ref{strsm} with  $ {\cal N}={\Bbb C}$.
 Our goal   is   to check the following result.

  %%%%%%%%%%%%%%%
\begin{theorem}\label{sm1}
%%%%%%%%%%%%%%%
Let the  operator $H_{\varphi, \theta}^{(0)}$ be    defined by  formula \eqref{eq:conf7}, and let the operator $Q_{a}$ be defined by formulas \eqref{eq:Xs} and \eqref{eq:X}.
Then the operator   $Q_{a}^{\beta } 
 Q_{\bar{a}}^{\beta }    $ for $\beta>1/2$ is strongly $H_{\varphi, \theta}^{(0)}$-smooth on $(-  1,  0)\cup (   0, 1)$ for the diagonalization $F Y_\varphi {\sf U} T_\alpha  $   with any exponent   $\gamma<\beta-1/2$.
 \end{theorem}
 
 \begin{proof}
 Let  us check the first estimate \eqref{eq:ZG}, that is, 
  \begin{equation}
|( F Y_\varphi {\sf U} T_\alpha  Q_{a}^{\beta } 
 Q_{\bar{a}}^{\beta }  f) (\tau)| \leq C \| f \|_{L^2(\Bbb T)}
\label{eq:X1C}
\end{equation}
for    $\tau$   in compact subintervals of ${\Bbb R}_{+}$ and all $f\in L^2 (\Bbb T)$. 
 Observe  that the operator $T_\alpha  Q_{a}^{\beta } 
 Q_{\bar{a}}^{\beta }  T_\alpha^*  $  acts as the  multiplication by a function bounded by
 $ C q_i (\mu )^{\beta}  q_{-i} (\mu )^{\beta}$.
 Therefore  the proof
  reduces to the case $a=i$ when   $T_\alpha =I$.   Next, we use that the  function $q_i (\mu )   q_{-i} (\mu )$ is even so that we can set $g(\mu^2)=  q_i (\mu ) q_{-i} (\mu ) $.  Let $G$ be the operator of multiplication by $g(\mu)  $.
 By the definition \eqref{eq:CF3y}, \eqref{eq:CF3} of the operator $\sf U$, we have  
$  {\sf U}Q_i^{\beta } Q_{-i}^{\beta } f=    G^\beta    {\sf U}f$. Since the operators $Y_{\varphi}$ and $G^\beta $ commute,
 estimate \eqref{eq:X1C} for $a=i$ can be rewritten as
   \begin{equation}
|( F G^\beta Y_\varphi {\sf U}  
   f) (\tau)| \leq C \| f \|_{L^2(\Bbb T)} .
\label{eq:X1CC}
\end{equation}
Note that the operator $Y_\varphi {\sf U} $ is unitary and that
  the function $g(\mu)$ is a bounded   by $C \big|  \ln |\mu+1 |  \big|^{-1}$ as $\mu\to - 1$.   
   Thus for the proof of \eqref{eq:X1CC}, it remains to use that according to Theorem~\ref{sm}  the operator $G^\beta $ is smooth with respect to $H(v)$
   (see estimate \eqref{eq:X1}). The second estimate \eqref{eq:ZG} can be verified quite similarly.
   \end{proof}

%%%%%%%%%%%%%%%%%%%%%%%%%%%%%%%%%%%%%%%%%%%%% 
%%%%%%%%%%%%%%%%%%%%%%%%%%%%%%%%%%%%%%%%%%%%% 
\section{Main results}\label{sec.g} 
%%%%%%%%%%%%%%%%%%%%%%%%%%%%%%%%%%%%%%%%%%%%%   
%%%%%%%%%%%%%%%%%%%%%%%%%%%%%%%%%%%%%%%%%%%%%

The main results are stated in subs.~1 and proven  in subs.~3. Necessary compactness results are collected in subs.~2. Then we reformulate our results   in subs.~4 in the representation ${\Bbb H}_{+}^2({\Bbb R})$ of Hankel operators.    The case of matrix-valued symbols is discussed in subs.~5.

 \medskip

{\bf 6.1.} 
Let $\omega \in L^\infty(\bbT)$ be a symbol   satisfying     condition 
\eqref{eq:HA2},  and let $H (\omega)$ be the corresponding self-adjoint  Hankel operator
\eqref{eq:HA} in ${\Bbb H}_{+}^2(\bbT)$. Our aim is to perform the spectral analysis of Hankel operators  with piecewise continuous symbols $\omega (\mu)$.

For a point $a \in \bbT$ of the discontinuity of $\omega (\mu)$, we define the jump by formula \eqref{eq:XX1} and assume condition \eqref{eq:XX3}.  
By \eqref{eq:HA2}, if $\omega (\mu)$ has a jump $\varkappa $ at some point $a \in{\Bbb T}$, then it also has
the jump $ -\bar{\varkappa}  $ at the point $\bar{a}$. In particular, the jumps at the points $\pm 1$ are purely imaginary. Let us write the points of discontinuity of $\omega$ as $  1$, $-1$, $a_{1},\ldots, a_{N_{0}}$, $\bar{a}_{1},\ldots, \bar{a}_{N_{0}}$, $\Im a_{j}> 0$, and the jumps of $\omega$ at these points as
\begin{equation}
\varkappa (\pm 1)= 2 i \kappa_{\pm }, \; \kappa_{\pm }\in{\Bbb R},\q
\varkappa (a_{j})=   2 \kappa_{j}e^{i \psi_{j}}, \; \kappa_{j}> 0 . 
\label{eq:HO2}\end{equation}
Of course, the  points $1$ or $-1$ may be regular; in this case condition \eqref{eq:XX3} at these points is not required and
$\kappa_+=0$ or $\kappa_{-}=0$. 

Thus, we accept

 \begin{assumption}\label{assuH}
%%%%%%%%%%%%%%%%%
A function $\omega:\bbT\to\C$ satisfies the self-adjointness condition 
\eqref{eq:HA2} and is  continuous 
apart from some jump discontinuities at finitely many points  $  1$, $-1$, $a_{1},\ldots, a_{N_{0}}$, $\bar{a}_{1},\ldots, \bar{a}_{N_{0}}$
with jumps \eqref{eq:HO2}. At every point of discontinuity $a \in \bbT$, it satisfies condition \eqref{eq:XX3}.  
\end{assumption}

Recall that ${\sf A}_{\Delta} $ is the operator of multiplication by independent variable in the space $L^2 (\Delta)$. We put 
\begin{equation}
\Delta_{\pm}=
[0,   \kappa_{\pm }] 
 \q {\rm and}
\q
 \Delta_j=[- \kappa_j,   \kappa_j] .
\label{eq:HOa}\end{equation}
The spectral structure of the operator $H=H(\omega)$ is described in the following assertion.

% which is equivalent to Theorem~\ref{XX} stated in the Introduction.

%\; {\rm if }\;\kappa_{\pm }>0, \q \Delta_{\pm}= [   \kappa_{\pm }, 0] \; {\rm if }\;     \kappa_{\pm }<0 \q

 \begin{theorem}\label{thm.g1a}
%%%%%%%%%%%%%%%%%
Let Assumption~$\ref{assuH}$ hold and $H=H(\omega)$. Then:
%, and let the intervals  $\Delta_{\pm}$ and $\Delta_{j}$ be defined by \eqref{eq:HOa}.  

  $1^0$
  If $\beta_{0}>1$, then the operator $H^{\rm (ac)}$ is unitarily equivalent to the orthogonal sum
  % The restriction of the operator $K   $ on the subspace ${\cal H}\ominus  {\cal H}^{(p)}(K)$  
  \begin{equation}
{\sf A}_{\Delta_{+}} \oplus {\sf A}_{\Delta_{-}} \oplus \bigoplus_{j=1}^{N_{0}} {\sf A}_{\Delta_{j} }.
\label{eq:AC1}\end{equation}
 % In particular, the absolutely continuous spectrum of  the operator $ K  $
%\begin{equation}
%\sigma^{(ac)}  (K  )= \Delta_{+}\cup \Delta_{-}\cup \bigcup_{j=1}^{N_{0}} \Delta_{j},
%\label{eq:AC2}\end{equation}

    $2^0$ If $\beta_{0}>2$, then  the singular continuous spectrum of  the operator $ H $ is empty and its eigenvalues, distinct from 
$0$,  $  \kappa_{\pm }$ and   
$\pm \kappa_{j}$, $j=1,\ldots, N_{0}$, have finite multiplicities and can accumulate only to these points.  
\end{theorem}

 Recall that, to  a given Hankel operator, there correspond various symbols
whose differences belong to the space $H^\infty_{-}(\Bbb T)$. Nevertheless the formulation of Theorem~\ref{thm.g1a} has an intrinsic nature because  
functions in  $H^\infty_{-}(\Bbb T)$ cannot have jumps.

%On the other hand, the above mentioned result about  functions from $H^\infty_{-}(\Bbb T)$ is a consequence of Theorem~\ref{thm.g1a}.  {\tt to check with N 3}

Next, we state   the limiting absorption principle for the operator $H $.

\begin{theorem}\label{thm.g1b}
%%%%%%%%%%%%%%%%%
Let Assumption~$\ref{assuH}$ hold with $\beta_{0}> 2$ and $H=H(\omega)$.
Let ${ Q}: H^2_{+} (\Bbb T)\to L^2  (\Bbb T)$ be the operator of multiplication by a bounded function ${ q}(\mu)$ such that
 \[
{ q}(\mu)= O(\big|\ln | \mu-a |\big| ^{-\beta}), \q \mu \to a,\q \beta> 1,
\]
%\label{eq:H2q}\end{equation}
  in all points of discontinuity of $\omega$, that is, in $a=1,-1, a_{1},\ldots, a_{N_{0}}, \bar{a}_{1},\ldots, \bar{a}_{N_{0}}$. 
 Then  
 the operator-valued function    ${  Q}( H  -z)^{-1}{ Q}^*$ is continuous in $z$ if $\pm \Im z\geq 0$ away from all points $0$,  $   \kappa_{\pm }$ and   
$\pm \kappa_{j}$, $j=1,\ldots, N_{0}$, and eigenvalues of the operator $H$. 
\end{theorem}

Finally, we consider the wave operators.
Recall that  the symbols $ v_{\pm}  $ were defined by equalities \eqref{H4},  \eqref{eq:om} and   \eqref{eq:pm}. The matrix symbol $ \Sigma^{(0)}_{\varphi}$ was defined by formula \eqref{eq:CF6}  and the   model operator $  H_{\varphi , \theta }^{(0)} $ was defined by  formulas   \eqref{eq:CF3},  \eqref{eq:conf} and \eqref{eq:conf7}.

 \begin{theorem}\label{thm.g1}
%%%%%%%%%%%%%%%%%
Let Assumption~$\ref{assuH}$ hold with $\beta_{0}> 1$ and $H=H(\omega)$.  Let the numbers $\kappa_{\pm}$, $\kappa_{j}$ and $\psi_{j}$ be defined by formula  \eqref{eq:HO2}. Put $a_{j}=e^{i\theta_{j}}$ and $\varphi_{j}= \psi_{j}+ \theta_{j}-\pi/2$.
Then all the assertions of Theorem~$\ref{ScTh}$ are true with  $N=N_{0}+2$ for  the operators  $H$  and $H_{j} =     \kappa_{j}  H_{\varphi_{j}, \theta_{j}}^{(0)}$    if 
 $j=1,\ldots, N_{0}$ and $H_{N_{0}+1}=     \kappa_{+} H(v_+)$, $H_{N_{0}+2}=     \kappa_{-} H(v_-)$.  
 \end{theorem}
 
In particular cases where a symbol $\omega$ has only one real singularity or only one pair of complex singularities, we have the following results.
 
 \begin{corollary}\label{cor1}
%%%%%%%%%%%%%%%%%
Suppose that a symbol $\omega(\mu)$ has only one jump $2 i\kappa_{+}$ at the point $  1$ or  $2 i\kappa_{-}$ at  $ - 1$. Then  the corresponding wave operators $W_{\tau}( H(\omega),    \kappa_{+} H(v_+))$ or $W_{\tau}( H(\omega),    \kappa_{-} H(v_-))$    $($for both signs $``\tau=\pm")$ exist and are complete.  
 \end{corollary}
 
 \begin{corollary}\label{cor2}
%%%%%%%%%%%%%%%%%
Suppose that a symbol $\omega(\mu)$ has only two jumps $2 \kappa e^{i\psi}$ and $- 2 \kappa e^{-i\psi}$,
$\kappa>0$, 
at  points $e^{i\theta}$ and $e^{-i\theta}$, respectively. Put $\varphi = \psi + \theta -\pi/2$. Then  the wave operators $W_{\pm}(H(\omega),   \kappa H_{\varphi, \theta }^{(0)}) $   exist and are complete.  
 \end{corollary}
 
Observe that $H(\Sigma_{\varphi, \theta}^{(0)})$ is not (if $\sin\varphi\neq 0$) a Hankel operator in the space ${\Bbb H}^2_{+}(\Bbb T)$. It is however possible to reformulate Theorem~\ref{thm.g1} solely in terms of Hankel operators. As usual, the symbols below satisfy condition \eqref{eq:HA2}.
 
  \begin{theorem}\label{G}
%%%%%%%%%%%%%%%%%
Let Assumption~$\ref{assuH}$ hold with $\beta_{0} >1$ and $H=H(\omega)$. Let $\varkappa (\pm 1)$, $\varkappa (a_{j})$, $\Im a_{j}>0$,  be the jumps of $\omega(\mu)$.
 Suppose that symbols $\omega_{\pm}$  satisfy Assumption~$\ref{assuH}$ and that their only jumps $\varkappa (\pm 1)$ are located at the points $\pm 1$. Suppose also that symbols  $\omega_{j}$, $j=1,\ldots, N_{0}$, satisfy Assumption~$\ref{assuH}$ and that their only jumps $\varkappa (a_{j})$  in the upper half-plane  are located at the points $a_{j}$. 
Then
 all the assertions of Theorem~$\ref{ScTh}$ with  $N=N_{0}+2$  are true   for the operators   $H_{j}=     H(\omega_{j})$, $j=1,\ldots, N_{0}$, $H_{N_{0}+1}=    H (\omega_{+})$, $H_{N_{0}+2}=    H(\omega_{-})$.  
 \end{theorem}

  \begin{proof}
  Let the index $\tau$ below take both values $``+"$ and $``-"$.
  It follows from Corollary~\ref{cor1} that the wave operators $W_{\tau}(H(\omega_{\pm}),  \kappa_{\pm} H(v_\pm))$   exist and are complete. Therefore, by the multiplication theorem for wave operators (see relation \eqref{eq:mt}),   the wave operators   $W_{\tau}(H,H(\omega_\pm) )$ also exist and
\begin{equation}
  W_{\tau}(H,H (\omega_\pm) )= W_{\tau}(H,  \kappa_{\pm} H(v_{\pm}) )W_{\tau}^*(  H(\omega_\pm),  \kappa_{\pm} H(v_{\pm}) ).
\label{eq:mt1}\end{equation}
   Similarly,   it follows from Corollary~\ref{cor2} that the wave operators $W_\tau (H(\omega_j),  \kappa_{j} H_{\varphi_{j} , \theta_{j} }^{(0)})$   exist and are complete. Therefore, by the multiplication theorem, the wave operators   $W_\tau(H,H(\omega_j) )$ also exist and
  \begin{equation}
  W_\tau(H,H(\omega_{j}) )= W_\tau(H, \kappa_{j}   H_{\varphi_{j} , \theta_{j} }^{(0)} )W_\tau^*(  H(\omega_{j}),   \kappa_{j} H_{\varphi_{j} , \theta_{j} }^{(0)}) .
  \label{eq:mt2}\end{equation}
Relations \eqref{eq:mt1} and \eqref{eq:mt2} imply that
\[
R\big(W_{\tau}(H,H(\omega_\pm) )\big)= R\big(W_{\tau}(H,  \kappa_{\pm} H(v_{\pm}) )\big)
\]
and
\[
 R\big( W_\tau(H,H(\omega_{j}) )\big)= R\big(W_\tau(H, \kappa_{j}   H_{\varphi_{j} , \theta_{j} }^{(0)} )\big).
\]
Therefore in the statement of Theorem~\ref{thm.g1},  the wave operators $W_{\tau}( H,  \kappa_{\pm} H(v_{\pm}) )$   can be replaced by $ W_{\tau}(H, H (\omega_{\pm}) )$, and
  the wave operators $W_\tau (H,  \kappa_{j}  H_{\varphi_{j} , \theta_{j} }^{(0)} )$ can be replaced by $ W_\tau (H,H(\omega_{j}) )$ .   
   \end{proof}
   
   Recall that the functions $\omega_{\varphi,\theta}$ defined by formulas  \eqref{eq:si},  \eqref{eq:conf4}  are smooth away from the points $e^{i\theta}$ and $e^{- i\theta}$ and have the jumps $2i e^{i(\varphi-\theta)}$ and $2i e^{i(\theta-\varphi)}$ at these points. Therefore for the symbol $\omega_{j}$ in  Theorem~\ref{G}, we can, for example, choose   the function
 \begin{equation}
 \omega_j(\mu)=   \kappa_{j}  \omega_{\varphi_{j},\theta_{j}}(\mu)
\label{eq:si1}\end{equation}
where    $a_{j} = e^{i \theta_{j}}$, $\varphi_{j} = \psi_{j} + \theta_{j} -\pi/2$. Similarly, since the functions $v_{\pm} (\mu)$ defined by formulas   \eqref{H4},  \eqref{eq:om}  and   \eqref{eq:pm} are smooth away from the points $\pm 1$ where they have the jump $2i$, we can set $\omega_{\pm} (\mu)= \kappa_{\pm} v_{\pm} (\mu)$. 

 It follows from Theorem~\ref{G} that, for all $f\in {\cal H}^{\rm(ac)}(H)$, the relation
 \[
 e^{-iHt}f= \sum_{j=1}^N  e^{-i H_{j}t}f_{j}^{(\pm)} +\varepsilon^{(\pm)} (t)
 \]
 holds with $f_{j}^{(\pm)}  =W_{\pm} (H,H_{j})^* f$ and $\varepsilon^{(\pm)} (t)\to 0$ as $t\to \pm \infty$. In particular, this relation shows that asymptotically the functions $( e^{-i Ht}f) (\mu)$ are concentrated for large $|t|$ in neighborhoods of singular points of the symbol $\omega(\mu)$. In a somewhat more simple situation,  this phenomenon was discussed in a detailed way in \cite{Y2}.
 
% We mention that all the results concerning the absolutely continuous spectrum and the wave operators remain true if   $\beta_{0}>1/2$ in Assumption~\ref{assuH}. The condition $\beta_{0}>1$ is used only for the study of the singular component of the spectrum.

% Actually, we were not able to find such model operators (which are explicitly diagonalisable) in the class of Hankel  operators; our model operators are ``almost Hankel''.

  \medskip

 %%%%%%%%%%%%%%%%%%%%%%%%%%%%%%%%%%%%%%%%%%%%% 
 {\bf 6.2.}
 In addition to the results of Sections~4 and 5 on model operators, 
  for the proof of   Theorems~\ref{thm.g1a}, \ref{thm.g1b} and \ref{thm.g1} we also need the results of \cite{PY1}, Section~4,   on the boundedness and compactness of Hankel operators sandwiched by singular weights. These results will be stated in this subsection. We emphasize that now $H(\omega) = P_{+} \Omega J P_{+}$ are considered as  operators in the space $L^2 ({\Bbb T})$.
Condition \eqref{eq:HA2} is supposed to be satisfied.

Recall that the function  $q_{a}  (\mu)$  is defined by equality \eqref{eq:Xs}. Let    $Q$ be
the operator of multiplication by  the   function  
\begin{equation}
q(\mu)
=  \Big(q_{1}  (\mu)  q_{-1}  (\mu)
\prod_{j=1}^N q_{a_{j}}  (\mu) q_{\bar{a}_{j}}  (\mu)\Big)^\beta,\q \beta>0, %\beta\in (1/2, \beta_{0}/2),
\label{eq:f6H}
\end{equation}
vanishing in all singular points of $\omega$. Of course $Q=Q^*$ and its kernel is trivial.

 The first assertion follows from the classical Muckenhoupt result \cite{Muck}.

\begin{lemma}\label{lma.f6}
%%%%%%%%%%%%%%%%%
Suppose   that   $\omega\in L^\infty(\Bbb T)$. 
Then the operators $Q  P_{+} \omega P_{-}Q^{-1}$ and hence $Q H(\omega) Q^{-1}$  
are bounded. 
\end{lemma}

The next statement generalizes the well known result about the compactness of Hankel operators $H(\omega)$ with $\omega\in C(\Bbb T) $.

 \begin{lemma}\label{CCCc}
%%%%%%%%%%%%%%%%%
Let   $\omega\in C(\Bbb T) $, and let condition \eqref{eq:XX3} with some $\beta_{0} >  2 \beta$ be satisfied in all singular points of $\omega$.
Then the operators $Q^{-1}  P_{+} \omega P_{-}Q^{-1}$ and hence $Q^{-1} H(\omega )  Q^{-1}$  
   are compact.
\end{lemma}

 \begin{lemma}\label{cross}
%%%%%%%%%%%%%%%%%
Let the symbols $\omega_{j}$ for $j=1,\ldots, N_{0}$ be defined by formula \eqref{eq:si1}, $\omega_{N_{0}+1}=\kappa_{+} v_{+}$ and $\omega_{N_{0}+2}=\kappa_{-} v_{-}$. Then, for  
$n,m= 1,\ldots, N_{0}+2$ and $n\neq m$,  all operators
 $Q^{-1}  P_{+} \omega_{n} P_{-} \omega_{m} P_{+}Q^{-1}$ and hence $Q^{-1} H(\omega_{n}) H(\omega_{m}) Q^{-1}$  
   are compact.
\end{lemma}

The proof of this result uses that the singularities of the symbols   $ \omega_{n}$ and $ \omega_{m}$ are disjoint if  $n\neq m$.

  \medskip

 %%%%%%%%%%%%%%%%%%%%%%%%%%%%%%%%%%%%%%%%%%%%% 
 {\bf 6.3.}
  For the proof of all Theorems~\ref{thm.g1a}, \ref{thm.g1b} and \ref{thm.g1}, we have to check Assumption~\ref{assu} with the operators  $H$, $H_n$, $n=1,\ldots, N=N_0 +2$, defined in Theorem~\ref{thm.g1} and  the operator  $\wt{H}$ defined by equality \eqref{eq:A}. For the smooth operator $Q$,  we   choose
the operator of multiplication by     function  \eqref{eq:f6H} where $\beta\in (1/2, \beta_{0}/2)$.  Since $\beta> 1/2$, 
Assumption~\ref{assu}a is satisfied with any $\gamma< \beta -1/2 $ according to Theorems~\ref{sm} and \ref{sm1}.

Let the jumps of $\omega(\mu)$ be given by formula \eqref{eq:HO2}. Define
    the functions $\omega_j $ for $j=1,\ldots, N_{0}$   by formula \eqref{eq:si1} and put $\omega_{N_{0}+1}=  \kappa_{+} v_{+}$,  $\omega_{N_{0}+2}=  \kappa_{-} v_{-}$.      Then the function
 \begin{equation}
 \ti{\omega}  (\mu)=  \omega (\mu) 
 -\sum_{n =1}^{N_{0}+2}  \omega_n (\mu) 
\label{eq:KK}\end{equation}
 has  no  jumps so that
$ \ti{\omega}  \in C (\Bbb T)$. Moreover, it follows from condition \eqref{eq:XX3} that 
\[
 \ti{\omega}  (\mu)-  \ti{\omega}  (a)= O \big(  \big|\ln |\mu-a|  \big|^{-\beta_{0}}  \big),\q \mu\to a,
\]
for all singular points $ a=\pm 1, a_{j}, \bar{a}_{j}$. Therefore  the operator $Q^{-1} H (  \ti{\omega}  ) Q^{-1} $ is compact according to Lemma~\ref{CCCc}.
  
 It follows from formula \eqref{eq:conf6} that
\begin{equation}
H(\omega_{j})
  = \kappa_{j}    H^{(0)}_{ \varphi_{j}, \theta_{j} }+ \kappa_{j}   \wt{H}_{ \varphi_{j}, \theta_{j} } , \q j=1,\ldots, N_{0},  
\label{eq:KK3}\end{equation}
where the operators $H^{(0)}_{ \varphi_{j}, \theta_{j} }$ and $ \wt{H}_{ \varphi_{j}, \theta_{j} }$ are defined by relations  \eqref{eq:conf7} and \eqref{eq:conf7B}, respectively. Since $\widetilde{\Sigma}_{\varphi_{j},\theta_{j} } \in C^\delta (\Bbb T)$ (for any $\delta<1$),  Lemma~\ref{CCCc} implies that the operators $Q^{-1}  \wt{H}_{ \varphi_{j}, \theta_{j} }   Q^{-1} $,  are compact.

Comparing definitions   \eqref{eq:A} and \eqref{eq:KK}, we see that
 \[
 \wt{H} = H ( \omega )  -\sum_{n=1}^{N_{0}+2} H_n =H (\ti{\omega}) +\sum_{n=1}^{N_{0}+2}(  H ( \omega_n ) -H_n ) 
 .
\]
%\label{eq:KK1}\end{equation}
Recall that $H ( \omega_n ) = H_n $ for $n=N_{0}+1, N_{0}+2$ and $H ( \omega_j )-H_{j}= \kappa_{j} \wt{H}_{\varphi_{j}, \theta_{j} }  $ for $j=1,\ldots, N_{0}$ according to \eqref{eq:KK3}.
Thus  the operator $Q^{-1} \wt{H} Q^{-1}$ is also compact which  verifies Assumption~\ref{assu}b.

To check Assumption~\ref{assu}c, we have to show that
the  operators 
\begin{equation}
Q^{-1} H (v_{+}) H (v_{-})  Q^{-1}, \q Q^{-1} H (v_{\pm})  H_j Q^{-1} \q {\rm and}\q 
Q^{-1} H_j H_l Q^{-1} ,
\label{eq:COMP}\end{equation}
where $  j,l=1,\ldots, N_{0}$, $ j\neq l$, 
  are compact.  For the first of these operators, this statement follows from Lemma~\ref{cross} because  the singularities of the symbols $v_{+}$ and $v_{-}$ (located at the points $1$ and $-1$) are separated. According to \eqref{eq:KK3}, we have
\[
H (v_{\pm} ) H_j =   H (v_{\pm} ) H (\omega_j) -  \kappa_{j} H (v_{\pm} )  \wt{H}_{ \varphi_{j}, \theta_{j} } .
\]
 The  operators $Q^{-1} H (v_{\pm} )  H (\omega_j) Q^{-1} $  are compact again by Lemma~\ref{cross} because  the singularities of the symbols $v_\pm $ and $\omega_j$
 (located at the points $\pm 1$ and $a_{j}, \bar{a}_{j}$)  are separated.  Observe that
\begin{equation}
Q^{-1} H (v_{\pm} )  \wt{H}_{ \varphi_{j}, \theta_{j} } Q^{-1} =(Q^{-1}  H (v_{\pm} ) Q) (Q^{-1} \wt{H}_{ \varphi_{j}, \theta_{j} } Q^{-1}).
\label{eq:COMPx}\end{equation}
In the right-hand side, the first factor  is bounded by Lemma~\ref{lma.f6} and
 the second factor  is compact by Lemma~\ref{CCCc}. Finally,   using again \eqref{eq:KK3} we find that
 \[
 H_j H_l= (H(\omega_{j}) -\kappa_{j}   \wt{H}_{ \varphi_{j}, \theta_{j} })(H(\omega_{l}) -\kappa_l  \wt{H}_{ \varphi_l, \theta_l }).
 \]
 The  operators $Q^{-1} H (\omega_j) H (\omega_l) Q^{-1}  $  are compact   because  the singularities of the symbols $\omega_j $ and $\omega_l$
 (located at the points $a_j, \bar{a}_j$ and $a_l, \bar{a}_l$)  are separated. The terms containing $\wt{H}_{ \varphi_{j}, \theta_{j} }$ or  $\wt{H}_{ \varphi_{l}, \theta_l }$ can be considered quite similarly to \eqref{eq:COMPx}. It follows that
  the third operator 
\eqref{eq:COMP} is also compact.

 Finally, Assumption~\ref{assu}d is satisfied according to Lemma~\ref{lma.f6}.  
 
   Thus we have verified  Assumption~\ref{assu} with the operators $H_n $, $n =1,\ldots, N=N_{0} +2 $, defined in Theorem~\ref{thm.g1}. Therefore 
 Theorems~\ref{thm.g1a}, \ref{thm.g1b} and \ref{thm.g1} are direct consequences of  Theorems~\ref{main}, \ref{LAP} and \ref{ScTh}, respectively.

  % The results of Section~4 and the proof in subs.~5.2 extend naturally to the case considered.

  \medskip
 
{\bf 6.4.} 
Let us now reformulate 
  the results of  subs.~6.1   in terms of Hankel operators  \eqref{eq:pww} acting in the space ${\cal H} = {\Bbb H}^2_{+}(\Bbb R)$. We recall that the operators ${\bf H}= {\bf H}(\psi)$   and $H=H(\omega)$ are unitarily equivalent (see formula \eqref{eq:pww}) if their symbols $\psi (\nu)$ and $\omega (\mu)$ are related by equality \eqref{eq:pwu}. If the symbol  $\omega (\mu)$ has a jump $\varkappa (1)$ or $\varkappa (-1)$ at  the point $+1$ or $-1$, then  the  symbol $\psi (\nu)$ has 
   the jump
   \[
   \pmb{\varkappa} (\infty):=\psi(+\infty)- \psi(-\infty)= - \varkappa (1)\q  \mbox{or} \q \pmb{\varkappa} (0):=\psi(+0)- \psi(-0)= - \varkappa (-1)
   \]
    at $\nu=\infty$ or $\nu=0$, respectively.
  If    $\omega (\mu)$ has   jumps $\varkappa  $ and $-\bar{\varkappa} $ at    points $a$ and $\bar{a}$, then    $\psi (\nu)$ has 
   the jumps $- a\varkappa  $ and   $  \bar{ a } \bar{\varkappa}$   at the points
 \[
 b=\frac{i}{2}\frac{1+a}{1-a}\q \mbox{and}\q - b= \frac{i}{2}\frac{1+\bar{a}}{1-\bar{a}}.
 \]
 %\label{eq:MA2}\end{equation}
 Note that $b<0$ if $\Im a>0$.

We assume that a symbol $\psi:\Bbb R\to\C$ satisfies the self-adjointness condition 
$\psi (-\nu) =\overline{\psi (\nu)} $ and is  continuous 
apart from some jump discontinuities at finitely many points    $0$ and $  b_{1}, -b_{1}, \ldots,   b_{N_{0}}, - b_{N_{0}}$ (we suppose that $b_{j}<0$) with jumps $ \pmb{\varkappa}(0) $ and
 $\pmb{\varkappa} ( b_{1}),  \pmb{\varkappa} (- b_{1})= -\ov{\pmb{\varkappa} ( b_{1})}, \ldots, \pmb{\varkappa} ( b_{N_{0}}), \pmb{\varkappa} (- b_{N_{0}})= -\ov{\pmb{\varkappa} ( b_{N_{0}})}$, respectively.  We also suppose that the limits $\psi(\pm\infty)$ exist and are finite.
At every singular point   $b  $, we assume the logarithmic H\"older continuity of $\psi(\nu)$. It is defined exactly as in  \eqref{eq:XX3} for finite $b$ and by the relation
 \[
\psi(\nu)- \psi(\pm \infty)= O(\big|\ln | \nu |\big| ^{- \beta_{0}}), \q \nu\to \pm \infty, \q\beta_{0}> 0,
\]
at infinity. 
 
Similarly to \eqref{eq:HO2}, we set
\begin{equation}
\pmb{\varkappa}(\infty) = 2 i\kappa_{\infty},\q \pmb{\varkappa} (0) = 2 i\kappa_{0}, \q
\pmb{\varkappa} (b_{j}) = 2 \kappa_{j}e^{i\phi_{j}},\q \kappa_{j} >0,  \q j=1,\ldots, N_{0},
\end{equation}
and
\begin{equation}
\Delta_{0}=[0, \kappa_{0} ], \q \Delta_{\infty}=[0, \kappa_{\infty} ], \q \Delta_{j}=[-\kappa_{j} , \kappa_{j} ].
\end{equation}
Then Theorem~\ref{thm.g1a} remains true  (with the same conditions on $\beta_{0}$) for the operator ${\bf H}$ if the numbers $\kappa_{+}$ and $\kappa_{-}$ are replaced by $-\kappa_{\infty}$ and $-\kappa_{0}$, respectively.  In particular,  the operator ${\bf H}^{\rm(ac)}$   is unitarily equivalent to the orthogonal sum
\begin{equation}
{\sf A}_{\Delta_{0}} \oplus {\sf A}_{\Delta_{\infty}} \oplus \bigoplus_{j=1}^{N_{0}} {\sf A}_{\Delta_{j} }.
\end{equation}

For the proof, it suffices to notice that under our assumptions on the symbol $\psi(\nu)$, the symbol  $\omega (\mu)$ defined  by equality \eqref{eq:pwu} satisfies Assumption~\ref{assuH}. Therefore Theorem~\ref{thm.g1a} applies to the operator $H(\omega)$ and we only have to use that ${\bf H} (\psi)= {\cal U} H(\omega)  {\cal U}^*$ where the unitary operator ${\cal U} $ is defined by \eqref{eq:un}.

 Theorem~\ref{thm.g1b} also remains unchanged if the function ${ q} (\mu)$ is replaced by a bounded function ${\bf q}(\nu)$ such that ${\bf q}(\nu) = O(\big| \ln |\nu-b|\big|^{-\beta})$ as $\nu\to b$ for some $\beta>1$ in all finite points $b$  of discontinuity   of $\psi(\nu)$ and ${\bf q}(\nu)= O(\big| \ln |\nu|\big|^{-\beta})$ as $|\nu|\to \infty$ if $\psi(+\infty)\neq \psi (-\infty)$.
 
 The reformulation of Theorem~\ref{G} for Hankel operators ${\bf H} ={\bf H} (\psi)$ in the space ${\Bbb H}^2_{+} ({\Bbb R})$ is quite obvious. Now we introduce symbols $\psi_{0}$, $\psi_{\infty}
$ and $\psi_{j}$, $j=1,\ldots, N_{0}$, with the same jumps as the symbol $\psi$ at  its singular points $0$, $\infty$ and $(b_{j}, - b_{j})$,  $j=1,\ldots, N_{0}$, respectively. Then again
 all the assertions of Theorem~$\ref{ScTh}$ with  $N=N_{0}+2$ are true    for the operators   $H_{j}=     H(\psi_{j})$, $j=1,\ldots, N_{0}$, and $H_{N_{0}+1}=    H (\psi_{\infty})$, $H_{N_{0}+2}=    H(\psi_{0})$.  
 
The symbols  $\psi_{0}$, $\psi_{\infty}
$ and $\psi_{j}$ can be expressed in terms of the function $\zeta (\nu)$    defined by equality \eqref{H4}:
\[
\psi_{0} (\nu) = \pmb\varkappa (0)\zeta(\nu), \q \psi_{\infty}(\nu)= - \pmb\varkappa (\infty)\zeta(-\nu^{-1})
\]
and
\[
\psi_j (\nu)= \pmb\varkappa (b_{j})   \zeta (\nu -b_{j}) +{\pmb\varkappa} (-b_{j}) \zeta (\nu + b_{j}), \q j=1,\ldots, N_{0}.
\]

Finally, we note that all model operators in the space ${\Bbb H}^2_{+} ({\Bbb T})$ can be transplanted into the space ${\Bbb H}^2_{+} ({\Bbb R})$ by the unitary transformation $\cal U$. This leads to the reformulation of Theorem~\ref{thm.g1}.

   \medskip
 
{\bf 6.5.}
  All our results can     be extended to Hankel operators acting in spaces of vector-valued functions.
Consider, for example, a Hankel operator $H(\omega)$    in the space ${\Bbb H}^2_{+} (\Bbb T)\otimes {\cal N}$ where ${\cal N}$ is an auxiliary Hilbert space and the symbol $\omega(\mu):{\cal N}\to {\cal N}$ is an operator-valued function.   We suppose that $\omega(\mu)$ are compact operators in ${\cal N}$ satisfying the condition $\omega(\bar{\mu})=\omega(\mu)^*$. Then the operator $H(\omega)$ is self-adjoint. The function $\omega (\mu)$ is supposed to be continuous in the operator norm apart from some jump singularities.   At singular points $1, -1, a_{1}, \bar{a}_{1}, \ldots, a_{N_{0}}, \bar{a}_{N_{0}}$, we     assume condition \eqref{eq:XX3}.  If ${\cal N}={\Bbb C}^k$, then $\omega(\mu)$ is of course a matrix-valued function.

At the points $\pm 1$, the function $\omega(\mu)$ may have  jumps $\varkappa (\pm 1)= 2 i K_{\pm} $ where $K_{\pm}$ are   self-adjoint operators in the space ${\cal N}$. In general , the operators  $K_{\pm}$  have both positive and negative eigenvalues denoted $  \kappa_{\pm}^{(1)},   \kappa_{\pm}^{(2)}, \ldots$.

  If $a_{j}$, $\Im a_j>0$, is a complex singular point of $\omega(\mu)$ with a jump  $\varkappa (a_{j})=2 K_{j}$,
  then   $\omega(\mu)$   also has the jump $-2 K_{j}^*$ at the conjugate point $\bar{a}_{j}$. Let us put  $K_{j}  = R_{j}+i S_{j}$ where $R_{j} =  R_{j}^*$, $S_{j} =  S_{j}^*$  and construct auxiliary compact self-adjoint operators
 \begin{equation}
{\sf K}_{j}= \begin{pmatrix} S_{j} & - R_{j}
 \\
 - R_{j} &  - S_{j}
  \end{pmatrix}
\label{eq:MA}\end{equation}
in the space ${\cal N}\otimes {\Bbb C}^2$.
Since ${\sf K}_{j} {\cal J} =-  {\cal J} {\sf K}_{j}$ for the involution
$
{\cal J} = \begin{pmatrix} 0  & i I
 \\
-i I& 0
  \end{pmatrix},
$
eigenvalues of the operators ${\sf K}_{j}$ are symmetric with respect to the point $0$. We write them as $\pm \kappa_{j}^{(1)}, \pm \kappa_{j}^{(2)}, \ldots$. Diagonalizing matrix \eqref{eq:MA}, we find that
\[
{\sf K}_{j}= Y_{j}^* \diag\{\kappa^{(1)}_{j}, -\kappa^{(1)}_j, \kappa^{(2)}_j, -\kappa^{(2)}_j, \ldots \} Y_{j}
\]
%\label{eq:MA1}\end{equation}
where $Y_{j}$ are unitary operators in the space $\cal N$. Of course, the explicit expression \eqref{eq:DD2} for these operators   no longer makes  sense.

%   $R_{j}+i S_{j}$, $R_{j}= R_{j}^*$, $S_{j}=S_{j}^*$. then   $\omega(\mu)$   also has the jump $-R_{j}+i S_{j}$ at the conjugate point $\bar{a}_{j}$.  

 Instead of \eqref{eq:HOa}, we now have the set of intervals
\[
\Delta_{\pm}^{(l)}=
[0,   \kappa_\pm^{(l)}]\q  {\rm and} \q\Delta_{j}^{(l)}=[-  \kappa_j^{(l)} ,   \kappa_j^{(l)}], \q   l=1, 2, \ldots.
\]
%\label{eq:HOa1}\end{equation}
Then Theorem~\ref{thm.g1a} remains true with the following natural modifications. In definition  \eqref{eq:AC1} one should set
\[
{\sf A}_{\Delta_{\pm} }=\bigoplus_{ l}{\sf A}_{\Delta_{\pm}^{(l)}} 
\q {\rm and} \q
  {\sf A}_{\Delta_j }=\bigoplus_{  l}  {\sf A}_{\Delta_j^{(l)} }
\]
which entails also an obvious modification of equality  \eqref{eq:XX4}.  Eigenvalues of  $ H (\omega)$, distinct from 
$0$,  $  \kappa_{\pm }^{(l)} $ and   
$\pm  \kappa_j^{(l)} $, $j=1,\ldots, N_{0}$, for all values of $l$, have finite multiplicities and can accumulate only to these points.  Theorem~\ref{thm.g1b} remains
unchanged.

Let us discuss a generalization of Theorem~\ref{thm.g1}. Similarly to the scalar case (cf. Section~4), the model operators corresponding to the jumps at the points $\mu=1$ and $\mu=-1$ can be defined by the formulas
\[
H_{N_{0} +1}=   H(v_{+})\otimes K_{+} \q {\rm and} \q H_{N_{0} +2}=   H(v_{-})\otimes K_{-}.
\]
 Then the proof of Theorem~\ref{sm} works directly because the operators $H(v_{\pm})$ and $K_{\pm}$  act in different spaces (${\Bbb H}_{+}^2 (\Bbb T)$ and $\cal N$, respectively).  Diagonalizing the self-adjoint operator $K_{\pm}$, we reduce the problem to the scalar case. 
    We emphasize that the space ${\cal N}$ entering into Definition~\ref{strsm} is the same as here.

The model operators $H_{j}$ corresponding to the jumps at the points $(a_{j}, \bar{a}_{j})$
are   constructed similarly to Section~5.   The role of   the model symbols $\kappa \omega_{\varphi} (\mu)$ and $\kappa \Sigma_{\varphi}^{(0)}(\mu)$ (recall formulas  \eqref{eq:si} and \eqref{eq:CF6}) for jumps at the points $(i-i)$ is now played by the matrix-values functions
\[
\omega_{K} (\mu)= (S-\mu R) v (\mu^2).
\]
and ${\sf K}  v(\mu)$, respectively. Therefore the model operators $H_{j}$ corresponding to the jumps at arbitrary complex points $(a_{j}, \bar{a}_{j})$ are constructed
by the formula  
\[
 H_j = T_{\alpha_{j}}^* {\sf U}^* {\sf H} ({\sf K}_{j}v){\sf U} T_{\alpha_{j}} .
  \]
This generalizes the ``scalar" formula $H_{j} =\kappa_{j}  H^{(0)}_{\varphi_{j}, \theta_{j}}$ where $H^{(0)}_{\varphi_{j}, \theta_{j}}$ are operators \eqref{eq:conf7}.

%%%%%%%%%%%%%%%%%%%%%%%%%%%%%%%%%%%%%%%%%%%%% 
%%%%%%%%%%%%%%%%%%%%%%%%%%%%%%%%%%%%%%%%%%%%% 
\section{Infinite matrices and integral operators}\label{sec.appl} 
%%%%%%%%%%%%%%%%%%%%%%%%%%%%%%%%%%%%%%%%%%%%%   
%%%%%%%%%%%%%%%%%%%%%%%%%%%%%%%%%%%%%%%%%%%%% 

The results of Section~6 can be reformulated in terms of Hankel operators acting in the spaces ${\ell}^2_{+}$ and $L^2 ({\Bbb R}_{+})$. This requires the Fourier expansion of symbols $\omega(\mu)$, $\mu\in {\Bbb T}$, and the Fourier transform  of symbols $\psi(\nu)$, $\nu\in {\Bbb R}$. Therefore the results stated in terms of matrix elements $h_{n}$
of operators $\wh{H}$ and of kernels $ {\bf h} (t)$  of operators $\wh{\bf H}$ are necessarily rather far from optimal. In notation of Section~3, we have  $h_{n}=\hat{\omega}_n$ and $ {\bf h} (t)=(2\pi)^{-1/2} \hat{\psi} (t)$.

 \medskip
 
{\bf 7.1.} 
Let us   consider the space ${\cal H} = {\ell}^2_{+}$ where a Hankel operator $\wh{H}$ acts by the formula \eqref{eq:HF} with matrix elements $h_{n}$.  
We study the case of matrix elements with asymptotics \eqref{eq:F}. To use the results of the previous section,  we have to  construct a symbol $\omega$ satisfying Assumption~\ref{assuH} and such that $\wh{H} ={\cal F} H (\omega){\cal F}^*$.
 Let us start with an elementary observation.

 \begin{lemma}\label{appl1}
%%%%%%%%%%%%%%%%%
Suppose that 
\[
\ti{h}_{n}=O\big(n^{-1}(\ln n)^{-\alpha}\big),\q \alpha>1.
\]
Then the function
\[
\ti{\omega}(\mu)=\sum_{n=0}^\infty \ti{h}_{n} \mu^n
\]
satisfies the logarithmic H\"older condition with   exponent $\beta = \alpha-1$, that is,
\begin{equation}
|\ti{\omega}(\mu')- \ti{\omega}(\mu)| \leq C \big(1+\big| \ln |\mu'-\mu| \big|\big)^{-\beta},\q \mu,\mu'\in {\Bbb T}.
\label{eq:app}\end{equation}
\end{lemma}

 \begin{proof}
 For all $N$, we have
 \begin{equation}
|\ti{\omega}(\mu')- \ti{\omega}(\mu)| \leq \sum_{n\leq  N}  |(\mu')^n-\mu^n||\ti{h}_{n}| 
 + \sum_{n> N}  |(\mu')^n-\mu^n|| \ti{h}_{n}| .
\label{eq:app1}\end{equation}
The first sum here is bounded by
\[
 \sum_{n\leq N} n  |\mu' -\mu | |\ti{h}_{n}|\leq
|\mu' -\mu | N \sum_{n=0}^\infty    |\ti{h}_{n}|.
 \]
 The second sum in \eqref{eq:app1} is bounded by
\[
2 \sum_{n>  N} |\ti{h}_{n}| \leq C \sum_{n>  N}n^{-1} (\ln n)^{-\alpha}\leq
C_{1} (\ln N)^{-\beta} .
 \]
 Choosing, for example, $N=  |\mu' -\mu |^{-1/2}$ and substituting these two estimates into \eqref{eq:app1},  we get \eqref{eq:app}.
   \end{proof}
   
   Let us   introduce the Hankel matrices $ \wh{H}_{+}$ and $ \wh{H}_-$  with elements
  \begin{equation}
     h_{n}^{(+)}=  \pi^{-1} (n+1)^{-1}\q{\rm and} \q  h_{n}^{(-)}= (-1)^n \pi^{-1} (n+1)^{-1},
 \label{eq:BB}\end{equation}
 respectively. As  can   easily be checked by a direct calculation, the   Fourier coefficients     of the function
   \begin{equation}
   \omega_{+}(\mu)= i    (1- \psi/\pi) e^{-i \psi},\q \mu =e^{i \psi},\q 0\leq \psi \leq 2\pi,
 \label{eq:app2}\end{equation}
   equal $h_{n}^{(+)}$ if $n\geq 0$. Similarly, the Fourier coefficients     of the function $\omega_{-}(\mu)= \omega_{+}(-\mu)$  equal $h_{n}^{(-)}$ if $n\geq 0$. It follows that  $\wh{H}_{\pm}={\cal F} H (\omega_{\pm}){\cal F}^*$. Note that $\omega_\pm(\bar\mu)=   \overline{\omega_\pm(\mu)}$. 
 %The following result is a consequence
 
  \begin{lemma}\label{Hi1}
%%%%%%%%%%%%%%%%%
$1^0$ The operators $ \wh{H}_\pm$ have the a.c. simple spectra coinciding with the interval $[0,1]$. They have no singular continuous spectra and their eigenvalues distinct from $0$ and $1$ have finite multiplicities and may accumulate to these points only.

$2^0$   Let the symbols $v_{\pm}$ be defined by equalities \eqref{H4}, \eqref{eq:om} and \eqref{eq:pm}. Put $\wh{H}(v_{\pm}) ={\cal F} H(v_{\pm}){\cal F}^*$. Then the wave operators   $W_{+}(   \wh{H}_{\pm},  \widehat{H}(v_{\pm}))$ and $W_{-}(   \wh{H}_{\pm},  \widehat{H}(v_{\pm}))$ exist and are complete. 
\end{lemma}

  \begin{proof}
  Consider, for example, the operators $\wh{H}_+$  and $\widehat{H}(v_{+})$.  Note that the functions $ \omega_{+}(\mu)$ and $ v_{+}(\mu)$ are smooth on ${\Bbb T}\setminus\{1 \}$ and     $\omega_{+}(1\pm i0)= v_{+}(1\pm i0)= \pm   i $. Therefore $\omega_{+}- v_{+}\in C^\delta ({\Bbb T})$ (for  any $\delta <1$) so that     Theorems~\ref{thm.g1a} and \ref{thm.g1} (see, in particular, Corollary~\ref{cor1}) apply to the operators $ H (v_{+})$ and $ H (\omega_{+})$.  Then it remains to transplant the results obtained into the space ${\ell }^2_{+}$ by the operator $\cal F$.
    \end{proof}

  \begin{remark}\label{Hi2}
%%%%%%%%%%%%%%%%%
Actually, the operators $ \wh{H}_\pm$ have no eigenvalues, and they can be explicitly diagonalized. Indeed, as shown in papers \cite{Ma,Ro}, the operator $\wh{\bf H}_{+}= {\cal L}\wh{H}_{+}{\cal L}^*$ is the Hankel integral operator in $L^2 ({\Bbb R}_{+})$ with kernel $h_{+} (t)=\pi^{-1} t^{-1}e^{-t}$; it can be  diagonalized in terms of the Whittaker functions.    The operator $ \wh{H}_-$ has the same properties since $ \wh{H}_-= \wh{R}^*  \wh{H}_+ \wh{R} $ where $\wh{R} $  is the unitary operator in ${\ell }_{+}^2$ defined by  $(\wh{R} u)_{n}= (-1)^n u_{n}$. 
\end{remark}

Next, we consider the Hankel  matrix $ \wh{H}_{\theta,\varphi}$ with elements
 \[
   h_{n} ( \theta,\varphi)=  2 \pi^{-1}  (n+1)^{-1} \sin (n \theta -\varphi ).
 \]
 % \label{eq:app3}\end{equation}
   In contrast to the operators $\wh{H}_{\pm}$, we cannot diagonalize the operators  $ \wh{H}_{\theta,\varphi}$  explicitly. Nevertheless similarly to Lemma~\ref{Hi1}, we can  describe the structure of their spectra.

    \begin{lemma}\label{appl2}
%%%%%%%%%%%%%%%%%
For all $\theta$ and $\varphi$, 
the operators $\wh{H}_{\theta,\varphi}$ have the a.c. simple spectra coinciding with the interval $[-1,1]$. They have no singular continuous spectra and their eigenvalues distinct from $0$, $1$ and $-1$ have finite multiplicities and may accumulate to these points only.  
    \end{lemma}

 \begin{proof}
 Set
 \[
 \omega_{\theta,\varphi} (\mu )= i \big(\omega_{+} (e^{- i\theta}\mu) e^{i\varphi}-
 \omega_{+} (e^{ i\theta}\mu ) e^{- i\varphi}\big)
 \]
 where the function $\omega _{+}$ is defined by equality \eqref{eq:app2}.
 The function $ \omega_{\theta,\varphi} (\mu  )$ satisfies Assumption~\ref{assuH}. It has only two points $e^{i\theta}$ and $e^{-i\theta}$ of discontinuity with the jumps $- 2 e^{i\varphi}$ and $ 2 e^{- i\varphi}$,
 respectively. Therefore Theorem~\ref{thm.g1a} (see, in particular, Corollary~\ref{cor2}) applies to the operator  $H (\omega_{ \theta,\varphi})$.

 Using expression  \eqref{eq:BB} for  the Fourier coefficients of the function $  \omega_+ (\mu ) $, we see that the Fourier coefficients of the function $  \omega_{\theta,\varphi} (\mu ) $ equal
 \[
 \hat{ \omega}_{n} ( \theta,\varphi)=i  (e^{i\varphi}e^{- i n\theta}-
 e^{-i\varphi}e^{ i n\theta}) h_{n}^{(+)}  =h_{n} ( \theta,\varphi),\q n\geq 0,
 \]
 and hence $ \wh{H}_{\theta,\varphi} ={\cal F} H ( \omega_{\theta,\varphi}  ){\cal F}^*$. 
  \end{proof}

 Let us return to the operator $\wh{H}$ whose matrix elements $h_{n}$ have asymptotics \eqref{eq:F} where $\alpha_{0} >2$.  Put
 \[
 \ti{h}_{n} = h_{n} -\kappa_{+}h_{n}^{(+)}-\kappa_{-}h_{n}^{(-)}- \sum_{j=1}^{N_{0}} \kappa_{j}h_{n} ( \theta_{j},\varphi_{j}), \q n\geq 0,  
 \]
 Since $\ti{h}_{n} =O(n^{-1} (\ln n)^{-\alpha_{0}})$ as $n\to\infty$, it follows from Lemma~\ref{appl1} that the function $\ti{\omega} ={\cal F}^* \ti{h}$ is logarithmic H\"older continuous with exponent $\beta_{0}=\alpha_{0}-1>1$. 
  Set
  \[
\omega (\mu)= \kappa_{+} \omega_{+}(\mu) + \kappa_{-} \omega_{-}(\mu) 
 +  \sum_{j=1}^{N_{0}} \kappa_{j} \omega_{\theta_{j}, \varphi_{j}} (\mu) +\ti{\omega}(\mu)  .
  \] 
 By our construction, the Fourier coefficients of this function are $({\cal F}\omega)_{n}= h_{n}$, $n\geq 0$. Let us now  apply the results of subs.~6.1 to the operators $H(\omega)$ and  $H_{j}= \kappa_{j} H (\omega_{\theta_{j}, \varphi_{j}})$, $j=1,\ldots, N_{0}$, $H_{N_{0}+1}= \kappa_{+} H (\omega_{+})$, $H_{N_{0}+2}= \omega_{-} H (\omega_{-})$. Transplanting these results into the space ${\ell}^2_{+}$ by the operator $\cal F$, we obtain the following assertion.

    \begin{theorem}\label{L}
%%%%%%%%%%%%%%%%%
Suppose that   coefficients $h_{n}$ of a Hankel matrix $\wh{H}$ admit representation \eqref{eq:F} where
   $\theta_{j} $ are distinct numbers in $ (0,\pi)$; the phases $\varphi_{j}\in [0,\pi)$  and the amplitudes 
   $\kappa_+, \kappa_-, \kappa_{j}\in{\Bbb R}$ are arbitrary.

  $1^0$
    If   $\alpha_{0}> 2$, then  the 
operator $\wh{H}^{\rm (ac) } $   is unitarily equivalent to the orthogonal sum \eqref{eq:AC1}.  Moreover, the wave operators $W_{\pm}(\wh{H}, \kappa_{\pm} \wh{H}_{\pm})$ and $W_{\pm}(\wh{H},  \kappa_{j} \wh{H} ( \theta_{j},\varphi_{j}))$, $j=1,\ldots, N_{0}$,  exist, their ranges are mutually orthogonal, and their orthogonal sum exhausts the subspace $   {\cal H}^{\rm (ac) }(\wh{H})$.
   
    $2^0$  If   $\alpha_{0}> 3$, then  the singular continuous spectrum of $\wh{H}$ is empty and its eigenvalues different from    the points $0$, $\kappa_{+} $,  $\kappa_{-}  $ and $ \pm \kappa_{j}$  have finite multiplicities and may accumulate only to these points. 
 \end{theorem}
 
 %  Of course, partially this result is contained in Theorem~\ref{XX1} stated in the Introduction.
   
 \medskip
 
{\bf 7.2.} 
Next, we consider a Hankel operator $\widehat{\bf H}$      acting in  the space ${\cal H} = L^2({\Bbb R}_{+})$ by the formula
\[
 (\widehat{\bf H} u )(t)=\int_{0}^\infty {\bf h}(t+s) u (s) ds.
\]
We suppose that   ${\bf h} \in L^1_{\rm loc} ({\Bbb R}_{+})$ and ${\bf h} (t)= O(t^{-1})$  as $t\to \infty$ and $t\to 0$. The operator 
$\widehat{\bf H}$ is symmetric if ${\bf h} (t)$ is a real function. Observe that
 $\widehat{\bf H}$ is compact if ${\bf h} (t)= o(t^{-1})$  as $t\to \infty$ and $t\to 0$. On the other hand,
if ${\bf h} (t)$ behaves as $t^{-1}$ for $t\to \infty$ or  for $t\to 0$, then the operator $\widehat{\bf H}$ acquires an a.c. spectrum. For example, the Mehler operator $\cal M$ defined by formula \eqref{eq:Meh}
has the simple a.c. spectrum coinciding with $[0,1]$.

We consider kernels with singularities both at $t=\infty$ and $t=0$.
 To be   precise, we assume that    
 \begin{equation}
 {\bf h} (t)=  (\pi t)^{-1} \big(h_{\infty}+ 2\sum_{j=1}^{N_{0}} h_{j}\sin ( b_{j} t- \phi_{j}) +O(| \ln t |^{-\alpha_{0}})\big),\q t\to \infty ,
  \label{eq:bb1}\end{equation}
  and
 \begin{equation}
 {\bf h}(t)= (\pi t)^{-1} \big(h_{0} + O(| \ln t |^{-\alpha_{0}})\big) ,\q   t\to 0.
  \label{eq:bb}\end{equation}
  We emphasize that the right-hand side of \eqref{eq:bb1} contains oscillating terms.

  % Observe that   mapping \eqref{eq:frlin} sends  the points $ -1 \in {\Bbb T}$ and $  1 \in {\Bbb T}$ into the points $0$ and $\infty$, respectively. The points $a=e^{i\psi}$ and  $\bar{a} $ are sent by this mapping into the points $\nu=-\cot (\psi/2)$ and $-\nu $, respectively.
  
  As in the previous subsection, we are going to use the results of Section~6 on Hankel operators in the Hardy space
  ${\Bbb H}^2_{+} ({\Bbb R} )$ (see subs.~6.4). Thus we have to construct a piecewise continuous symbol $\psi(\nu)$
  such that $\widehat{\bf H}=\Phi {\bf H} (\psi) \Phi^*$.
  The following assertion plays the role of Lemma~\ref{appl1}.

 \begin{lemma}\label{appl1L}
%%%%%%%%%%%%%%%%%
Suppose that a function $ \widetilde{\bf h} \in L^1 ({\Bbb R}_{+})$ obeys the condition
\begin{equation}
\widetilde{\bf h}(t)=O(t^{-1}|\ln t|^{-\alpha}),\q \alpha>1,
\label{eq:appL}\end{equation}
  as $t\to\infty$.
Then the function
\[
\ti{\psi}(\nu)=\int_0^\infty \widetilde{\bf h}(t) e^{i \nu t} dt
\]
%\label{eq:appLz}\end{equation}
is  logarithmic H\"older continuous with   exponent $\beta = \alpha-1$, that is,
\begin{equation}
|\ti{\psi}(\nu')- \ti{\psi}(\nu)| \leq C (1+\big| \ln |\nu'-\nu| \big|)^{-\beta},\q \nu,\nu'\in {\Bbb R}.
\label{eq:appL1}\end{equation}
\end{lemma}

 \begin{proof}
It follows from condition \eqref{eq:appL} that for an arbitrary $a>1$
\begin{equation}
\big| \int_0^a  \widetilde{\bf h}(t) ( e^{i \nu' t} - e^{i \nu t}) dt \big|  \leq  C |\nu'-\nu|
\int_{0}^a t  |\widetilde{\bf h}( t ) | dt 
%\nonumber\\
\leq  C_{1}  |\nu'-\nu| \, a \int_{0}^\infty  | \widetilde{\bf h}( t ) | dt  .
\label{eq:appL2}\end{equation}
Moreover, we have
\[
\big| \int_{a }^\infty \widetilde{\bf h}(t) ( e^{i \nu' t} - e^{i \nu t}) dt \big|  \leq 2  \int_{a }^\infty | \widetilde{\bf h}(t) | dt \leq C |\ln a|^{-\beta}.
\]
 Combining this estimate with \eqref{eq:appL2} and choosing, for example,  $a=|\nu'-\nu|^{-1/2}$,
we get \eqref{eq:appL1}.
   \end{proof}

Let us first consider kernels ${\bf h} (t)$ with asymptotics    \eqref{eq:bb1} as $t\to\infty$ and regular at the point $t=0$.
Recall that,    as shown in subs.~4.2,   the symbol $\psi_{\infty}(\nu)$ of the Mehler operator $ {\cal M}=:\widehat{\bf H}_{\infty}$  can be chosen as $\psi_{\infty}(\nu) = 2 i   \zeta (\nu)$ where $\zeta (\nu)$ is function \eqref{H4}. It follows that
the Fourier transform of the function %\eqref{eq:MA3}
\[
\psi_{\phi,b}(\nu)= 2 e^{-i\phi} \zeta (\nu +b)- 2 e^{i\phi} \zeta (\nu -b)
\]
 equals
$  (2\pi)^{1/2}  {\bf h}_{\phi ,b}(t)$   where
\[
{\bf h}_{\phi,b}(t)= 2 \pi^{-1} (2+t)^{-1}   \sin (bt-\phi).
\]
Thus,  a symbol of the operator $ \widehat{\bf H}_{\phi,b}$ with integral kernel  ${\bf h}_{\phi,b}(t)$ can be chosen as $\psi_{\phi,b}(\nu)$. Since the function $\psi_{\phi,b}(\nu)$ has only two jumps $- 2 e^{i\phi}$ and $ 2 e^{ -i \phi}$ at the points $b$ and $-b$, respectively, Theorem~\ref{thm.g1a} entails the following result (cf. Lemma~\ref{appl2}).

 \begin{lemma}\label{appl2a}
%%%%%%%%%%%%%%%%%
For all $\phi$ and $b\neq 0$,   the 
operators $  \widehat{\bf H}_{\phi,b}$ have the a.c. simple spectra coinciding with the interval $[-1,1]$.   They have no singular continuous spectra and their eigenvalues distinct from $0$, $1$ and $-1$ have finite multiplicities and may accumulate to these points only.  
    \end{lemma}

    It remains to construct a model operator for the kernel with singularity \eqref{eq:bb} as $t\to 0$. Observe that the Fourier transform of the function
    \[
    \psi_{0} (\nu)= 2 \pi^{-1} i \int_{0}^\infty \frac{ \sin (t\nu)}{t} e^{-t} dt= \pi^{-1}\mbox {v.p.} \int_{-\infty}^\infty
    \frac{ e^{ it\nu} } {t} e^{-| t |} dt
 \]
 %  \label{eq:appL4}\end{equation}
 (the right integral is understood in the sense of the principal value)    equals
    \[
\hat \psi_{0} (t)=(2/\pi)^{1/2}  \mbox {v.p.} \,  t^{-1} e^{-| t |}.
\]
This implies that $ \psi_{0} (\nu)$   is a symbol of the Hankel operator $\widehat{\bf H}_{0}$ with kernel ${\bf h}_{0}(t)=(\pi t)^{-1}e^{-t}  $. The function $ \psi_{0} (\nu)$ is smooth, but its limits at infinity $ \psi_{0} (\pm \infty)=\pm i$ are different. Thus by Corollary~\ref{cor1}, the 
operator $ \widehat{\bf H}_{0}$  has the a.c. simple spectrum coinciding with the interval $[0,1]$.   It  has no singular continuous spectrum  and its eigenvalues distinct from $0$ and $1$   have finite multiplicities and may accumulate to these points only.    Actually, it is known (see Remark~\ref{Hi2}) that the   operator $ \widehat{\bf H}_0$ has no eigenvalues, and it can be explicitly diagonalized.
 
  Let us return to the operator $\widehat{\bf H}$ whose kernel has asymptotics \eqref{eq:bb1} and \eqref{eq:bb}. Put
   \[
 {\bf h}_{\infty}(t)= h_{\infty} \pi ^{-1} (t+2)^{-1}, \q  {\bf h}_j(t)= 2 h_{j}\pi ^{-1} (t+2)^{-1}\sin ( b_{j} t- \phi_{j}), \q {\bf h}_{0}(t)= h_{0}(\pi t)^{-1} e^{-t}
  \]
  and
   \begin{equation}
 \widetilde{\bf h}(t)=  {\bf h}(t)-  {\bf h}_{\infty}(t)-\sum_{j=1}^{N_{0}} {\bf h}_j(t) -  {\bf h}_{0}(t).
\label{eq:MN6}\end{equation}
  By our construction, $\ti{{\bf h}} \in L^1 ({\Bbb R}_{+})$ and it satisfies condition \eqref{eq:appL} with $\alpha=\alpha_{0}$ both as $t\to \infty$ and $t\to 0$. It follows from Lemma~\ref{appl1L} that the function
   \begin{equation}
  \ti{\psi} (\nu)= \int_{0}^\infty   \widetilde{\bf h}(t) e^{i\nu t} dt
  \label{eq:MN7}\end{equation}
  is  logarithmic H\"older continuous with   exponent $\beta_{0} = \alpha_{0}-1$. Of course $ \ti{\psi} (\nu)\to 0$ as $|\nu|\to \infty$, but to verify the condition
 \begin{equation}
  \ti{\psi} (\nu)=O( \big|\ln |\nu|\big|^{-\beta_{0} })  \q{\rm as}  \q |\nu|\to \infty,
\label{eq:MN2}\end{equation}
we need an additional (very weak) assumption.

 \begin{assumption}\label{bad}
%%%%%%%%%%%%%%%%%
A function ${\bf h}(t)$ is absolutely continuous except a finite number of jumps and, for some $k$,
\[
\int_{a^{-1}}^a |{\bf h}'(t) | dt =O(a^{k}) \q{\rm as}  \q a\to \infty.
\]
%\label{eq:MN1}\end{equation}
\end{assumption}
  
 \begin{lemma}\label{appl1Lz}
%%%%%%%%%%%%%%%%%
Let   conditions \eqref{eq:bb1},  \eqref{eq:bb} and Assumption~$\ref{bad}$  be satisfied. Then
 function \eqref{eq:MN7} obeys   estimate \eqref{eq:MN2} with
  $\beta_{0}=\alpha_{0}-1$.
\end{lemma}
  
    \begin{proof} 
    By definition \eqref{eq:MN6}, the function $ \widetilde{\bf h}(t)$ also satisfies Assumption~$\ref{bad}$.
    Integrating by parts, we see that
    \[
\big| \int_{a^{-1}}^a \widetilde{\bf h}(t) e^{i \nu t} dt\big|\leq C |\nu|^{-1} a^{k}.
\]
%\label{eq:MN3}\end{equation}
where $a\to \infty$. It follows from estimate    \eqref{eq:appL} on $ \widetilde{\bf h}(t)$ for $t\to 0$ and $t\to \infty$  that
 \[
\big| \int_0^{a^{-1}}  \widetilde{\bf h} (t) e^{i \nu t} dt\big|+ \big| \int_a^{\infty}  \widetilde{\bf h}(t) e^{i \nu t} dt\big| \leq C |\ln a|^{-\beta_{0}} .
\]
%\label{eq:MN4}\end{equation}
Choosing, for example, $a= |\nu|^{1/ (2k)}$, we obtain \eqref{eq:MN2}.
        \end{proof} 
        
        Let us now put 
         \[
\psi(\nu)= { h}_{\infty} \psi_{\infty}(\nu) + \sum_{j=1}^{N_{0}} { h}_{j} \psi_{\phi_{j}, b_{j}}(\nu) 
+ {  h}_{0} \psi_{0}(\nu) + \ti{\psi} (\nu).
\]
%\label{eq:MN5}\end{equation}
It follows from formula \eqref{eq:MN6} that
 $\widehat{\bf H}=\Phi {\bf H}(\psi) \Phi^*$. Now we   can apply the results of subs.~6.4 to the operators
${\bf H}(\psi)$ and $h_{\infty}  {\bf H}(\psi_{\infty})$, $h_j  {\bf H}(\psi_{\phi_{j}, b_{j}})$, $h_0  {\bf H}(\psi_0)$. This yields the following result.

 \begin{theorem}\label{LL}
%%%%%%%%%%%%%%%%%
Let $\wh{\bf H}$ be the Hankel operator with   kernel  ${\bf h} \in L^1_{\rm loc} ({\Bbb R}_+)$
satisfying conditions \eqref{eq:bb1}, \eqref{eq:bb} and Assumption~$\ref{bad}$. We
suppose that   the numbers $b_1,\ldots, b_{N_{0}}\in {\Bbb R}_+\setminus \{0\}$ are distinct, the phases $\phi_j \in [0, \pi) $, $j=1,\ldots, N_0$, as well as the amplitudes $h_n \in {\Bbb R}$, $n=1,\ldots, N$, are arbitrary.

  $1^0$
 If   $\alpha_{0}> 2$, then  the 
operator $\wh{\bf H}^{\rm (ac) } $   is unitarily equivalent to the orthogonal sum
\begin{equation}
{\sf A}_{(0, h_{0})} \oplus {\sf A}_{(0, h_{\infty})} 
  \oplus \bigoplus_{j=1}^{N_{0}} {\sf A}_{(-h_{j}, h_{j})} .
\label{eq:MN2s}\end{equation}
Moreover, the wave operators $W_{\pm}(\wh{\bf H}, h_{0} \wh{\bf H}_{0})$, $W_{\pm}(\wh{\bf H}, h_{\infty} \wh{\bf H}_{\infty})$
  and $W_{\pm}(\wh{\bf H}, h_{j} \wh{\bf H} ( \theta_{j},\varphi_{j}))$, $j=1,\ldots, N_{0}$,  exist, their ranges are mutually orthogonal, and their orthogonal sum exhausts the subspace $   {\cal H}^{\rm (ac) }(\wh{\bf H})$.

    $2^0$ If   $\alpha_{0}> 3$, then  the singular continuous spectrum of $\wh{\bf H}$ is empty and its eigenvalues different from     the points $0$, $h_0 $,  $h_{\infty} $ and $ \pm h_{j}$  have finite multiplicities and may accumulate only to these points. 
 \end{theorem}
 
 \begin{remark}\label{LLR}
  If there is no singularity at the point $t=0$, that is, ${\bf h} \in L^1  (0,r)$ for   $r<\infty$, then  condition  \eqref{eq:bb}   and Assumption~\ref{bad} disappear and the term ${\sf A}_{(0, h_{0})} $ in \eqref{eq:MN2s} should be omitted. Indeed, in this case the point $\nu=\infty$ is not singular for the symbol of the operator  $\Phi^* \widehat{\bf H}\Phi  $ so that we do not need to verify \eqref{eq:MN2} 
 and hence  condition  \eqref{eq:MN7} is not required.
  \end{remark}

 \begin{remark}\label{LLY}
 The case $h_{0}=h_{\infty}$, $h_{j}=0$ for all $j=1,\ldots, N_{0}$ was considered in \cite{Y2}. It was shown there that the assertion $1^0$ of Theorem~\ref{LL} holds true for $\alpha_{0}>1$ and the assertion $2^0$  -- for $\alpha_{0}>2$. 
 Assumption~\ref{bad} was also not required. Therefore it can  be expected that
 the conditions of Theorem~\ref{LL} are  not  optimal. The same remark applies to Theorem~\ref{L}.
 We also note that using the Mourre method  J.~S.~Howland \cite{Howland} obtained the spectral results (but not  the results about the wave operators) of  Theorem~\ref{LL} assuming that  $h_{j}=0$ for all $j=1,\ldots, N_{0}$ but admitting  that 
 $h_{0}\neq h_{\infty}$.
 \end{remark}

 As far as earlier results about Hankel operators $\widehat{\bf H}$ with oscillating kernels are concerned, we mention the paper \cite{KM}  where the case ${\bf h}(t)= 2 (\pi t)^{-1} \sin (bt)$ was considered.  Obviously, for different $b>0$, these operators are unitarily equivalent to each other.
 With a help of the results of \cite{Ro}, it was proven in \cite{KM}   that the spectrum of such $\widehat{\bf H}$ is  a.c.  simple   and coincides with the interval $[-1,  1]$. This result is of course consistent with  Theorem~\ref{LL}.

  %  We also assume some local regularity of the function $h(t)$. 
%    To be precise, we require that $h \in L^1_{\rm loc} ({\Bbb R}_{+})$ and
 %  \begin{equation}
%\int_{t_{0}}^{t_{1}} h(t) e^{i\nu t} dt=   O(| \ln \nu |^{-\beta_{0}}\big) ,\q   \beta_{0} >2, 
%  \label{eq:LRe}\end{equation}
% as $\nu\to\infty$ for all $t_{0},t_{1} \in {\Bbb R}_{+}$.

%%%%%%%%%%%%%%%%%%%%%%%%%%%%%%%%%%%%%%%%
%%%%%%%%%%%%%%%%%%%%%%%%%%%%%%%%%%%%%%%%

\end{document}